\documentclass[12pt,english,reqno]{amsart}
\usepackage[T1]{fontenc}
\usepackage[latin9]{inputenc}
\usepackage[a4paper]{geometry}
\usepackage{amstext}
\usepackage{amsthm}
\usepackage{amssymb}
\usepackage{setspace}
\usepackage{enumerate}
\makeatletter
\numberwithin{equation}{section}
\numberwithin{figure}{section}
\theoremstyle{plain}
\newtheorem{thm}{\protect\theoremname}[section]
\theoremstyle{remark}
\newtheorem{rem}[thm]{\protect\remarkname}
\theoremstyle{definition}
\newtheorem{defn}[thm]{\protect\definitionname}
\theoremstyle{plain}
\newtheorem{lem}[thm]{\protect\lemmaname}
\theoremstyle{plain}
\newtheorem{prop}[thm]{\protect\propositionname}
\theoremstyle{plain}
\newtheorem{cor}[thm]{\protect\corollaryname}
\theoremstyle{remark}
\newtheorem{observation}[thm]{\protect\observationname}
\theoremstyle{remark}
\newtheorem{claim}[thm]{\protect\claimname}
\theoremstyle{plain}

\usepackage{amsmath}
\usepackage{bbm}
\usepackage{dsfont}

\usepackage[unicode=true,pdfusetitle,
 bookmarks=true,bookmarksnumbered=false,bookmarksopen=false,
 breaklinks=false,pdfborder={0 0 0},pdfborderstyle={},backref=false,colorlinks=false]
 {hyperref}

\usepackage[dvipsnames]{xcolor}
\usepackage{tikz}
\usetikzlibrary{arrows,decorations.markings}
\usetikzlibrary{patterns}

\usepackage{etoolbox}
\patchcmd{\@settitle}{\uppercasenonmath\@title}{}{}{}
\patchcmd{\@setauthors}{\MakeUppercase}{}{}{}
\patchcmd{\section}{\scshape}{}{}{}

\makeatother

\usepackage{babel}
\providecommand{\claimname}{Claim}
\providecommand{\corollaryname}{Corollary}
\providecommand{\definitionname}{Definition}
\providecommand{\factname}{Fact}
\providecommand{\lemmaname}{Lemma}
\providecommand{\observationname}{Observation}
\providecommand{\propositionname}{Proposition}
\providecommand{\remarkname}{Remark}
\providecommand{\theoremname}{Theorem}

\allowdisplaybreaks

\begin{document}
\global\long\def\One{\mathds{1}}%

\global\long\def\Laplacian{\Delta}%

\global\long\def\grad{\nabla}%

\global\long\def\norm#1{\left\Vert #1\right\Vert }%

\global\long\def\zz{\mathbb{Z}}%

\global\long\def\rr{\mathbb{R}}%

\global\long\def\nn{\mathbb{N}}%

\global\long\def\pp{\mathbb{P}}%

\global\long\def\ee{\mathbb{E}}%

\global\long\def\floor#1{\left\lfloor #1\right\rfloor }%

\global\long\def\ceil#1{\left\lceil #1\right\rceil }%

\global\long\def\var{\operatorname{Var}}%

\global\long\def\dd{\operatorname{d}}%

\global\long\def\loss{\operatorname{Loss}}%

\global\long\def\dom{\operatorname{Dom}}%

\global\long\def\etab{\overline{\eta}}%

\global\long\def\mub{\overline{\mu}}%

\global\long\def\cb{\overline{c}}%

\global\long\def\eb{\overline{e}}%

\global\long\def\etah{\hat{\eta}}%

\global\long\def\muh{\hat{\mu}}%

\title{Hydrodynamic limit for the Kob-Andersen model}
\author{Assaf Shapira}
\email{\href{mailto:assaf.shapira@normalesup.org}{assaf.shapira@normalesup.org}}
\urladdr{\href{https://assafshap.github.io/}{assafshap.github.io}}
\begin{abstract}
This paper concerns with the hydrodynamic limit of the Kob-Andersen
model, an interacting particle system that has been introduced by
physicists in order to explain glassy behavior, and widely studied
since. We will see that the density profile evolves in the hydrodynamic
limit according to a non-degenerate hydrodynamic equation, and understand 
how the diffusion coefficient decays as density grows.
\end{abstract}

\maketitle

\section{Introduction}

The Kob-Andersen (KA) model is an interacting particle system on $\zz^{d}$,
where each site of the lattice is allowed to contain at most one particle,
and particles could jump to an empty neighboring site only under a
certain constraint, conserving the total number of particles.
More precisely, depending on an integer parameter $k$,
every particle jumps with rate $1$ to each of its neighboring sites,
provided that the particle has at least $k$ empty neighbors both
before and after the jump (so for $k=1$ we obtain the symmetric simple
exclusion process). This model has been introduced in the physics
literature (\cite{KobAndersen}) as one member of a large family of
interacting particle systems called \emph{kinetically constrained
lattice gases} (KCLGs), which model certain aspects of glassy behavior
(see \cite{GarrahanSollichToninelli2011,RitortSollich}). 

In this paper we will study the hydrodynamic limit of the KA model.
Consider a finite box with periodic boundary conditions $\mathbb{T}_{N}^{d}=\zz^{d}/N\zz^{d}$,
and run the KA dynamics inside $\mathbb{T}_{N}^{d}$. The configuration
at time $s$ could be described as an empirical measure $\nu_{s}^{(N)}$
on the continuous torus $\mathbb{T}^{d}=\rr^{d}/\zz^{d}$: for a rectangle
$R\subset\left[0,1\right]^{d}$, seen as a subset of $\mathbb{T}^{d}$,
$\nu_{s}^{(N)}(R)$ will count the number of particles in $(NR)\cap\mathbb{T}_{N}^{d}$,
normalized by $N^{-d}$ (so that the total mass remains independent
of $N$). The initial configuration that we choose will be approximated
by some macroscopic profile $\rho_{0}:\mathbb{T}^{d}\rightarrow\left[0,1\right]$,
i.e., the measure $\nu_{0}^{(N)}$ will be close to a measure $\nu_{0}$
that has density $\rho_{0}$ with respect to the Lebesgue measure.
A simple example of such initial configuration is given by placing
a particle at each site $x\in\zz^{d}$ independently at random with
probability $\rho_{0}(x/N)$.

In many systems, the relevant time scale over which $\nu_{N}^{s}$
changes macroscopically is the diffusive time scale $N^{2}$ (see,
e.g., \cite{KipnisLandim,Spohn2012IPS}). That is, fixing a time $t$,
we expect the random measure $\nu_{N^{2}t}^{(N)}$ to satisfy a law
of large numbers, converging to some limiting measure
$\nu_{t}$. We also expect this limiting measure to have a density
with respect to the Lebesgue measure, namely $\nu_{t}=\rho(\theta,t)\text{d}\theta$,
which solves the diffusion equation: 
\begin{equation}
\frac{\partial}{\partial t}\rho=\grad\,D(\rho)\grad\rho,\quad\rho(\theta,0)=\rho_{0}(\theta).\label{eq:HE}
\end{equation}
The parameter $D(\rho)$ is the \emph{diffusion coefficient}, and
when it is non-zero we obtain indeed a macroscopic density profile
that changes over diffusive time scales.

Hydrodynamic limits of other KCLGs have been analyzed in \cite{GoncalvesLandimToninelli,BlondelGoncalvesSimon}.
They present two example of \emph{non-cooperative} KCLGs, in which
one is able to identify structures (called \emph{mobile clusters})
that could move freely in $\zz^{d}$ (see Definition \ref{def:mobile_cluster}).
This way, even though particles could be blocked, mobile clusters behave
effectively in an unconstrained manner.
In \emph{cooperative} KCLGs there are no such mobile clusters,
so in order to move a particle from one site to the other one needs
the cooperation of a diverging number of particles. This property
has a major contribution to the glassy behavior of many KCLGs, and
is responsible for the fast divergence of time scales.
See \cite{ToninelliBiroliFisher}.

Unlike the models previously studied in \cite{GoncalvesLandimToninelli,BlondelGoncalvesSimon},
the KA model is cooperative. Due to this cooperative nature, the combinatorics
behind the KA model becomes much more complicated. Consider the following
question -- starting from a stationary measure and assuming that
there is a particle at the origin, will this particle eventually move,
or could it stay at the origin forever? When the model is non-cooperative
the probability to stay forever at the origin is clearly $0$ --
we know that there is some non-zero density of mobile clusters in
$\zz^{d}$ which diffuse freely, so at some point one of them will
reach the origin and move the particle. When the model is cooperative,
as in the case of the KA model, already this basic question becomes
much more complicated. In some cooperative models the particle might
remain blocked forever with positive probability, possibly depending
on the density of the initial configuration. In the case of the KA
model, it is shown in \cite{ToninelliBiroliFisher} that all particles
will eventually move with probability $1$, unless the initial density
equals $1$.

In the context of the hydrodynamic limit, the techniques used in \cite{GoncalvesLandimToninelli,BlondelGoncalvesSimon}
cannot be simply adapted to cooperative models. It is shown in Appendix
\ref{sec:appendix} that cooperative KCLGs are non-gradient, a fact
which makes the analysis of the hydrodynamic limit much more involved.
See Appendix \ref{sec:appendix} for a more complete discussion.
Another property of non-cooperative models used in \cite{BlondelGoncalvesSimon}
is that the probability for a site to stay blocked forever for the
dynamics in $\mathbb{T}_{N}^{d}$ decreases exponentially fast with
the volume $N^{d}$, since it is bounded by the probability that no
mobile cluster is found in $\mathbb{T}_{N}^{d}$. In the KA model,
on the other hand, even though this probability decays to $0$, the
decay is not fast enough. 

Recently, a few methods have been developed to overcome some of these
difficulties, proving diffusive scaling of the relaxation time \cite{MST19KA}
and of the motion of a tagged particle \cite{BlondelToninell2018KAtagged,ES20KAtagged}
in the stationary setting. In both cases, the behavior is the same
as that of the simple exclusion process, with time scales that are
all slowed down by a factor which diverges quickly as the density
approaches $1$. For example, in the case $k=d=2$, the relaxation
time at density $\rho$ in a box of side $N$ behaves (roughly) like
$e^{C/(1-\rho)}N^{2}$; and the path of a tagged particle in $\zz^{2}$
converges to a standard Brownian motion as the length scale $N$ diverges,
when time is scaled (roughly) as $e^{C/(1-\rho)}N^{2}$.

The hydrodynamic limit of the KA model has been studied in the physics
literature, both heuristically and numerically. In \cite{Sellitto2002}
the model has been analyzed, under the (wrong) assumption that the
diffusion coefficient $D(\rho)$ vanishes for $\rho>\rho_{c}\approx0.88$.
\cite{TeomyShokef2017KAHL} studies the diffusion coefficient in two
dimensions both numerically and under a mean-field approximation.
This approximation yields a diffusion coefficient that behaves polynomially
in $\rho$, and is in rather good agreement with numerical results
for low densities. \cite{AritaKrapivskyMallick2018} provides a perturbative
analysis of the diffusion coefficient in two dimensions, considering
finite range effects, and obtaining a polynomial in $\rho$ which
approximates $D(\rho)$ very accurately as long as $\rho$ is not
too big. In view of other quantities related to the KA model studied
in \cite{MST19KA,ES20KAtagged}, a natural conjecture for the high
density regime is that the diffusion coefficient remains positive
whenever $\rho<1$, and as $\rho$ tends to $1$ it decays (roughly)
as $e^{-C/(1-\rho)}$ (in the case $k=d=2$). This conjecture has
been raised in \cite{AritaKrapivskyMallick2018} and was supported
by numerical simulations.

The hydrodynamic limit in its full generality, though, cannot exist
for this model -- consider, for example, the case $k=d=2$, and an
initial density $\rho_{0}$ bounded above $\frac{8}{9}$. Fix $N\in3\nn$,
and construct the following initial configuration -- for every $x\in\mathbb{T}_{N}^{2}$,
if $x\notin3\zz^{2}$ place a particle at $x$ (deterministically).
Otherwise, place a particle at $x$ independently at random with probability
$9\rho_{0}(x/N)-8$. See Figure \ref{fig:blocked_configuration}. These configurations have limiting density $\frac{1}{9}\left(9\rho_{0}(x/N)-8\right)+\frac{8}{9}=\rho_{0}$,
so one may naively expect that, starting the KA dynamics from such
a configuration, the particle density will converge to the solution
of the hydrodynamic equation (\ref{eq:HE}) with initial density $\rho_{0}$.
However, observing the initial configuration more carefully, one sees
that it is blocked -- no site has two empty neighbors, so the constraint
is not satisfied. In this case particles do not move, and the dynamics
will certainly not follow the hydrodynamic limit. Still, since blocked
configurations are very rare (\cite{ToninelliBiroliFisher}), we may
hope that a hydrodynamic limit does exist in a weaker sense, that
would allow us to avoid these untypical configurations.

\begin{figure}
\begin{tikzpicture}[scale=0.5, every node/.style={scale=1.0}]
	\draw[step=1,gray] (0,0) grid +(9,9);
	
	\foreach \x in {0,...,8}{
		\foreach \y 
			[evaluate=\y as \z using {mod(\y,3)==1 && mod(\x,3)==1}]
			 in {0,...,8}{
			\ifnum \z=0 
				\fill[black,shift={(0.5,0.5)}] (\x,\y) circle (0.4);
			\else
				\fill[pattern=north east lines,shift={(0.5,0.5)}] (\x,\y) circle (0.4);
			\fi
		}
	}
		
\end{tikzpicture}
\caption{\label{fig:blocked_configuration}This is an example of a blocked configuration for the case $k=d=2$. The filled circles represent occupied sites, while sites marked with a line pattern could be either empty or occupied. In this example particles have at most one empty neighbor, so none of them could move.}
\end{figure}

The same problem also appears in \cite{GoncalvesLandimToninelli},
and they suggest two solutions -- the first is to restrict the initial
configuration, e.g., to an independent product of Bernoulli random
variables with parameter $\rho_{0}(x/N)$. This prevents the issue
discussed above, where the configuration is entirely blocked from
the beginning, but one must work harder in order to show that blocked
configurations are not created later on during the dynamics. Another
approach, also considered in \cite{GoncalvesLandimToninelli}, is
to permit transitions in which the constraint is not satisfied, but
with a vanishing rate. Namely, for some $\varepsilon>0$, we introduce
soft constraints, which allow a particle to move with rate $1$ when
it has $k$ empty neighbors before and after the jump, and with rate
$\varepsilon$ otherwise. This softening of the constraint enables
the system to unblock the blocked configurations, and still the main
contribution to the overall dynamics comes from the allowed transitions
(where the constraint is satisfies).

This is the approach we will take -- consider the KA model with $\varepsilon$-soft
constraints, which has a hydrodynamic limit with diffusion coefficient
$D^{(\varepsilon)}$. We analyze this coefficient, showing that, as
$\varepsilon\rightarrow0$, it converges to a strictly positive limiting
coefficient $D$. This result tells us that when $\varepsilon$ is
very small, it has a very mild effect on the hydrodynamic limit; and
the role it plays (of unblocking configurations), though crucial for
the convergence to the hydrodynamic limit, takes a negligible amount
of time compared to the hydrodynamic scale. We also analyze the value
of $D$ at large densities, finding upper and lower bounds for its
decay, which match up to sub-leading corrections. The decay that we
obtain is of the same type as the corresponding factor in \cite{MST19KA,ES20KAtagged};
so in particular for the case $k=d=2$, as conjectured is \cite{AritaKrapivskyMallick2018},
$D$ decays (roughly) as $e^{-C/(1-\rho)}$.

\section{Model and main result}

The Kob-Andersen model in dimension $d$ is a Markov process on $\Omega=\{0,1\}^{\zz^{d}}$,
depending on a parameter $2\le k\le d$. For a configuration $\eta\in\Omega$,
we say that $x\in\zz^{d}$ is \emph{occupied} if $\eta(x)=1$ and
\emph{empty} if $\eta(x)=0$. The elements of $\zz^{d}$ are called
\emph{sites}, and we will consider the (undirected) graph structure
given by the edge set 
\[
\mathcal{E}(\zz^{d})=\left\{ (x,y)\in\zz^{d}\times\zz^{d},y\in x+\{\pm e_{1},\dots,\pm e_{d}\}\right\} ,
\]
where $e_{1},\dots,e_{d}$ are the standard basis vectors. We will
sometimes write $x\sim y$ to denote $(x,y)\in\mathcal{E}(\zz^{d})$.

For each configuration $\eta\in\Omega$ and edge $(x,y)\in\mathcal{E}(\zz^{d})$,
we define the constraint
\begin{equation}
c_{x,y}(\eta)=\begin{cases}
1 & \text{if } \sum_{z:y\sim z\neq x}(1-\eta(z))\ge k-1\text{ and }\sum_{z:x\sim z\neq y}(1-\eta(z))\ge k-1,\\
0 & \text{otherwise}.
\end{cases}\label{eq:constraint}
\end{equation}
The KA dynamics is then defined as the Markov process whose generator,
operating on a local function $f:\Omega\rightarrow\rr$, is given
by 
\begin{equation}
\mathcal{L}f(\eta)=\sum_{(x,y)\in\mathcal{E}(\zz^{d})}c_{x,y}(\eta)\grad_{x,y}f(\eta),\label{eq:generator}
\end{equation}
where 
\[
\grad_{x,y}f(\eta)=f(\eta^{x,y})-f(\eta),
\]
and $\eta^{x,y}$ is the configuration obtained from $\eta$ by exchanging
the occupation at $x$ and at $y$. This process, for any $\rho\in(0,1)$,
is reversible with respect to the measure $\mu_{\rho}$, which is
a product measure of Bernoulli random variables with parameter $\rho$.
This is a consequence of the fact that $c_{x,y}(\eta)$ does not depend on
the occupation at $x$ and at $y$.
When clear from the context we will sometimes omit the subscript $\rho$. 
For more details on the construction of the model see \cite{CMRT2010}.

As discussed in the introduction, in order to study the hydrodynamic
limit we introduce the \emph{soft constraint} for some $\varepsilon \ge 0$:
\begin{equation}
c_{x,y}^{(\varepsilon)}=\begin{cases}
1 & \text{if } c_{x,y}=1,\\
\varepsilon & \text{otherwise},
\end{cases}\label{eq:softconstraint}
\end{equation}
and the \emph{soft dynamics} defined by the generator 
\begin{equation}
\mathcal{L}^{(\varepsilon)}f(\eta)=\sum_{(x,y)\in\mathcal{E}(\zz^{d})}c_{x,y}^{(\varepsilon)}(\eta)\grad_{x,y}f(\eta).\label{eq:softgenerator}
\end{equation}

The introduction of the soft constraints allows us to use the general
result of \cite{FunakiUchiyamaYau, VaradhanYau97}. Fix $\varepsilon \ge 0$, and let 
\begin{equation}
D^{(\varepsilon)}(\rho)=\frac{1}{2\rho(1-\rho)}\inf_{f}\mu_{\rho}\left[\sum_{\alpha}c_{0,e_{\alpha}}^{(\varepsilon)}\left(\delta_{\alpha,1}\left(\eta(e_{1})-\eta(0)\right)-\sum_{x\in\zz^{d}}\grad_{0,e_{\alpha}}\tau_{x}f\right)^{2}\right],\label{eq:vy}
\end{equation}
where the infimum is taken over all local functions $f:\Omega\rightarrow\rr$;
a function is \emph{local} if it depends on the occupation of finitely many
sites. Note that the sum over $x$ is well defined thanks to the locality of $f$.
The operator $\tau_{x}$ is the translation by $x$, that is, 
\begin{align*}
(\tau_{x}f)(\eta) & =f(\tau_{x}\eta),\\
(\tau_{x}\eta)(y) & =\eta(x+y).
\end{align*}
 In this setting, by  \cite{FunakiUchiyamaYau,Bernardin02,VaradhanYau97}, the density profile of
the soft dynamics converges in the hydrodynamic limit to the solution
of the hydrodynamic equation (\ref{eq:HE}), with diffusion coefficient
$D^{(\varepsilon)}(\rho)$.

By equation (\ref{eq:vy}) the diffusion coefficient is decreasing
with $\varepsilon$, and hence converging to a limit: 
\begin{equation}
D(\rho)=\lim_{\varepsilon\rightarrow0}D^{(\varepsilon)}(\rho).\label{eq:D_eps0}
\end{equation}
When taking $\varepsilon$ to $0$ slowly enough as $N$ grows to infinity, the
density profile converges to the solution of the diffusion equation \eqref{eq:HE}
with this diffusion coefficient:
\begin{prop} \label{prop:epsilon_decreases_with_N}
Fix a smooth initial density profile $\rho_0 : \mathbb{T}^d \to (0,1)$, and let
$\nu_0$ be the measure whose density with respect to the 
Lebesgue measure is  $\rho_0$.
Consider a sequence of initial conditions $(\eta_0^{(N)})_{N\in \nn}$, with
$\eta_0^{(N)} \in \{0,1\}^{\mathbb{T}^d_N}$, such that the associated
empirical measures $\nu_0^{(N)}$ converge to $\nu_0$. 
Let $\rho_t(\theta)$ be the solution of the diffusion equation $\eqref{eq:HE}$, with
diffusion coefficient $D$ given by equation \eqref{eq:D_eps0}, and $\nu_t$ the measure
with density $\rho_t$.
For $s\ge0$, denote by $\nu_s^{(\varepsilon,N)}$ the (random) empirical measure associated with
the Kob-Andersen model on $\mathbb{T}^d_N$ with $\varepsilon$-soft constraints
at (microscopic) time $s$, with the initial configuration $\eta_0^{(N)}$.

Then there exists a sequence $(\varepsilon_N)_{N\in \nn}$ for which $\nu_{N^2 s}^{(\varepsilon_N,N)}$ converges
in probability to $\nu_t$ as $N$ tends to infinity.
\end{prop}
\begin{rem}\label{rem:isotropic}
In general, the diffusion coefficient is a matrix given by (see \cite[Propositions 2.1 and 2.2]{Spohn2012IPS})
\[
D_{\alpha\beta}=\lim_{t\rightarrow\infty}\frac{1}{t}\,\frac{1}{2\rho(1-\rho)}\sum_{x\in\zz^{d}}x_{\alpha}x_{\beta}\left(\mu_{\rho}(\eta(0)\,e^{t\mathcal{L}}\eta(x))-\rho^{2}\right).
\]
The reason that $D^{(\varepsilon)}(\rho)$ in equation (\ref{eq:vy})
is a real number, is that in our case $D$ is a scalar matrix: the
dynamics is invariant under inversion of a single coordinate (i.e.,
$x\mapsto x-(2x\cdot e_{\alpha})\,e_{\alpha}$), and therefore, if
$\alpha\neq\beta$, the sum $\sum_{x\in\zz^{d}}x_{\alpha}x_{\beta}(\mu(\eta(0)\,e^{t\mathcal{L}}\eta(x))-\rho^{2})$
must vanish. That is, $D$ is a diagonal matrix. Since the dynamics
is also invariant under permutation of coordinates, all diagonal elements
are equal, i.e., $D$ is scalar. This fact is useful for the analysis
of the limiting PDE in the proof of Proposition \ref{prop:epsilon_decreases_with_N}.
\end{rem}

The main result of this paper is that $D$ is strictly positive,
so that the hydrodynamic limit is not degenerate, i.e., the density
profile evolves over diffusive time scales.
\begin{thm}
\label{thm:main}For all $\rho\in(0,1)$, 
\begin{align*}
D(\rho)\ge & \begin{cases}
C/\exp\left(\lambda\log(1/(1-\rho))^{2}\,(1-\rho)^{-1/(d-1)}\right) & k=2,\\
C/\exp^{k-1}\left(\lambda (1-\rho)^{-1/(d-k+1)}\right) & k\ge3,
\end{cases}\\
D(\rho)\le & C'/\exp^{k-1}(\lambda'(1-\rho)^{-1/(d-k+1)}),
\end{align*}
where $\exp^{k}(\cdot)$ is the $k$-th iterate of the exponential.
The constants $C,C',\lambda,\lambda'$ are all strictly positive,
and may depend only on $d$ and $k$.
\end{thm}

\section{Proof of Proposition \ref{prop:epsilon_decreases_with_N}}

The proof is based on the results of \cite{FunakiUchiyamaYau,Bernardin02,VaradhanYau97},
together with the continuity of the solution $\rho_t(\theta)$ with respect to the diffusion coefficient.

Denote by $\rho^{(\varepsilon)}_t$ the solution of the diffusion equation
\eqref{eq:HE} with diffusion coefficient $D^{(\varepsilon)}$ given by
equation \eqref{eq:vy}. The existence and uniqueness of $\rho_t$ and
$ \rho_t^{\varepsilon}$, as well the maximum principle, comes from the
theory of parabolic equations (see, e.g., \cite{Vazquez07PME})
\footnote{Most works treat the equation on $\rr^d$ rather than
the torus. Nonetheless, the results we need hold also for  the equation
on $\mathbb{T}^d$, see the discussion in Section 11.5 of \cite{Vazquez07PME}.}.
Indeed, since $D$ is bounded and isotropic (see Remark \ref{rem:isotropic}), we 
may define a continuous increasing positive function 
$\Phi(u) = \int_0^u D(\rho) \dd \rho$, 
allowing us to write equation
 \eqref{eq:HE} as
\[
\partial_t \rho = \Laplacian \Phi(\rho),
\]
known as the \emph{generalized porous medium equation}, or the
\emph{filtration equation}. It is the subject of \cite{BenilanCrandall81}, 
and discussed thoroughly in \cite{Vazquez07PME}.

The main tool we use is:
\begin{thm}[\cite{FunakiUchiyamaYau,Bernardin02,VaradhanYau97}] \label{thm:HL}
For any smooth test function $f$ on $\mathbb{T}^d$,
\[
\int f(\theta) \dd \nu^{(\varepsilon,N)}_{N^2 t}(\theta) \xrightarrow {N\to\infty} \int f(\theta) \rho^{\varepsilon}_t(\theta) \dd \theta
\]
in probability.
\end{thm}

In addition, we need to know that, for small $\varepsilon$, the profile
$\rho^{(\varepsilon)}_t$ is close to $\rho_t$. This problem is analyzed
in \cite{BenilanCrandall81} in a much more complicated setting, where $D(\rho)$
may approach $0$ in some points of space. Since we only consider the
case where $\rho$ is bounded away from $1$, the assumptions of \cite{BenilanCrandall81}
are easily verified, yielding:
\begin{claim} \label{claim:solution_continuity}
For all $t>0$
\[
\rho^{(\varepsilon)}_t \xrightarrow {\varepsilon \to 0} \rho_t
\]
in $L^1(\mathbb{T}^d)$.
\end{claim}

In order to prove Proposition \ref{prop:epsilon_decreases_with_N}, we will fix
a dense countable family $\{f_m\}_{m\in\nn}$ of bounded functions on $\mathbb{T}^d$. Then,
as discussed in \cite[Chapter 4.1]{KipnisLandim}, it suffices to show that
\begin{equation}
P \left[
\left | \int f_m(\theta) \dd \nu_{N^2 t}^{(\varepsilon_N,N)}(\theta)
-
\int f_m(\theta) \rho_t(\theta) \dd \theta \right | > \delta
\right]
\xrightarrow {N \to \infty} 0
\end{equation}
for any fixed $\delta>0$ and all $m\in \nn$, for an appropriately chosen
sequence $\{\varepsilon_N\}$.

An immediate corollary of Theorem \ref{thm:HL} and Claim \ref{claim:solution_continuity} is:
\begin{cor}
Fix $M > 0$. Then there exists $\varepsilon(M),N(M)$ such that,
for all $m\le M$ and $N \ge N(M)$,
\[
P \left[
\left | \int f_m(\theta) \dd \nu_{N^2 t}^{(\varepsilon(M),N)}(\theta)
-
\int f_m(\theta) \rho_t(\theta) \dd \theta \right |
> \delta \right ]
\le \frac{1}{M}.
\]
Moreover, we may assume $\varepsilon(M) \to 0$ and $N(M) \to \infty$ as $M \to \infty$. 
\end{cor}
\begin{proof}
By Claim \ref{claim:solution_continuity}, for all $m$ 
there exists $\varepsilon^*(m)$ such that
\[
\left | \int f_m(\theta) \rho_t^{(\varepsilon)}(\theta) \dd \theta
-
\int f_m(\theta) \rho_t(\theta) \dd \theta \right | < \delta/2
\]
for all $\varepsilon < \varepsilon^*(m)$.
We will choose $\varepsilon(M) = \min_{m<M} \varepsilon^*(m) \wedge \frac{1}{M}$.

By  Theorem \ref{thm:HL}, for any $m$ and any $\varepsilon$ there exists $N(m,\varepsilon)$,
such that if $N \ge N(m,\varepsilon)$ then
\[
P \left[
\left | \int f_m(\theta) \dd \nu_{N^2 t}^{(\varepsilon,N)}(\theta)
-
\int f_m(\theta) \rho_t^{(\varepsilon)}(\theta) \dd \theta \right |
> \delta /2 \right ]
\le \frac{1}{M}.
\]
Define $N(M) = \max_{m \le M} N(m,\varepsilon(M)) \vee M$.
This concludes the proof of the corollary.
\end{proof}

We are now ready to choose our sequence $\varepsilon_N$:
\begin{align*}
M_N &= \max \{M : N \ge N(M) \},\\
\varepsilon_N &= \varepsilon(M_N).
\end{align*}
Then, indeed,
\[
P \left [
\left | \int f_m(\theta) \dd \nu_{N^2 t}^{\varepsilon_N,N}(\theta)
-
\int f_m(\theta) \rho(\theta) \dd \theta \right |
> \delta \right ]
\le \frac{1}{M_N}
\]
for all $m \le M_N$. This concludes the proof of the proposition. \qed

%
%
%

\section{\label{sec:lowerbound}Proof of the lower bound}

The purpose of this section is to prove 
\begin{align}
D^{(0)} & \ge L^{-\lambda},\label{eq:lowerbound}\\
L & =\begin{cases}
C\exp\left(\lambda\log(1/q)^{2}\,q^{-1/(d-1)}\right) & k=2,\\
C\exp^{k-1}\left(\lambda q^{-1/(d-k+1)}\right) & k\ge3,
\end{cases}\label{eq:lowerboundL}
\end{align}
where for convenience we denote $1-\rho=q$.
Throughout the section $\lambda$ and $C$ denote generic positive
constants, depending only on $k$ and $d$, that may be updated
from one line to the other. This will prove the first
inequality of Theorem \ref{thm:main} since $D\ge D^{(0)}$.

The proof is based on a comparison to the diffusion coefficient of
a random walk on the infinite component of a percolation cluster.
The idea behind the proof, is that even though at small scale particles
are blocked, at a large scale there is high probability that somewhere
a droplet containing many empty sites could approach the particle
allowing it to move; and this is the scale which determines the diffusion
coefficient. This mechanism is constructed in \cite{MST19KA,ES20KAtagged}
using the notion of a \emph{multistep move} -- a sequence of exchanges,
all allowed for the KA dynamics, moving a particle with the aid of
a nearby vacancies.

We start by providing the exact definition of a multistep move (see
also \cite{MST19KA}):
\begin{defn}[multistep move]
Fix $\mathcal{M}\subseteq\Omega$ and $T\in\nn$. A \emph{$T$-step move}
$M$ with \emph{domain} $\mathcal{M}$ is a function from $\mathcal{M}$ to $\left(\Omega\times\zz^{d}\times\{\pm e_{1},\dots,\pm e_{d},0\}\right)^{T+1}$,
described by a sequence of functions $M=\left\{ \eta_{t}(\eta),x^{t}(\eta),e^{t}(\eta)\right\} _{t=0}^{T}$,
such that, for all $\eta\in\mathcal{M}$,
\begin{enumerate}
\item $\eta_{0}(\eta)=\eta$,
\item for all $t\in\{1,\dots,T\}$, $\eta_{t}(\eta)=\eta_{t-1}(\eta)^{x^{t},x^{t}+e^{t}}$,
\item for all $t\in\{1,\dots,T\}$, $c_{x^{t},x^{t}+e^{t}}(\eta_{t}(\eta))=1$,
where by convention we set $c_{x,x}(\eta)=1$ for all $x,\eta$.
\end{enumerate}
\end{defn}
\begin{description}
\item[Warning] $t$ in the above definition does \emph{not} stand
for the (continuous) time in which the process evolves. It is a 
discrete variable indexing steps in a multistep move.
\end{description}
\begin{defn}
Fix a $T$-step move $M$ with domain $\mathcal{M}$. Then, for $t\in\{1,\dots,T\}$,
the \emph{loss of information} at time $t$, denoted $\loss_{t}(M)$,
is defined as 
\[
2^{\loss_{t}(M)}=\sup_{\eta'\in\mathcal{M}}\#\{\eta\in\mathcal{M}:\eta_{t}(\eta)=\eta_{t}(\eta'),x^{t}(\eta)=x^{t}(\eta'),e^{t}(\eta)=e^{t}(\eta')\}.
\]
We also set $\loss(M)=\sup_{t}\loss_{t}(M)$.
\end{defn}

The multistep move that we will define will allow us to move a particle
at $x$ to the site $x+Le_{\alpha}$ ($\alpha\in\{1,\dots,d\}$).
The choice of $L$ in equation (\ref{eq:lowerboundL}) guarantees
that such a multistep move could indeed be applied. Note that
$C, \lambda$ can be chosen such that $L \in \nn$.

We will therefore consider the coarse grained lattice $\zz_{L}^{d}=L\zz^{d}$,
and split a generic configuration $\eta$ in two -- the occupation of the
sites of $\zz_{L}^{d}$ denoted $\etab\in\overline{\Omega}_L=\left\{ 0,1\right\} ^{\zz_{L}^{d}}$,
and that of the sites outside $\zz_{L}^{d}$ denoted $\etah\in\left\{ 0,1\right\} ^{\zz^{d}\setminus\zz_{L}^{d}}$.
We will also split the measure in two, such that $\etab$ distributes
according to $\mub$ and $\etah$ according to $\muh$. Note that both measures $\mub,\muh$
are Bernoulli product measure with parameter $\rho$ (which is implicit in the notation).
The coarse grained lattice has a graph structure (isomorphic to $\zz^{d}$),
i.e., two vertices $i,j$ are connected by an edge if $i-j\in\left\{ \pm\eb_{1},\dots,\pm\eb_{d}\right\} $,
where $\eb_{\alpha}=Le_{\alpha}$. We denote the edge set by $\mathcal{E}(\zz_{L}^{d})$. 

Next, we will define a multistep move allowing particles to move on the coarse
grained lattice $\zz_{L}^{d}$, that is, it will exchange the occupation
at $Li$ with the occupation at $Lj$ for some edge $(i,j)\in\mathcal{E}(\zz_L^d)$.
In order to achieve that, there must be sufficiently many empty sites organized
in a proper fashion in the region between $i$ and $j$. We can think of an edge
in $\mathcal{E}(\zz_{L}^{d})$ satisfying this condition as \emph{open}, and otherwise
\emph{closed}, defining a percolation process on $\zz_L^d$. An important property
of this construction is that the conditions we require for an edge to be open or
closed will only depend on the occupation \emph{outside} $\zz_L^d$, namely $\etah$.
\begin{lem}
\label{lem:multistepmove}There exist a percolation process $\cb(\etah)\in\Pi=\{0,1\}^{\mathcal{E}(\zz^{d}_L)}$
and $T$-step moves $M^{\pm\eb_{1}}, \dots,  \linebreak[1] M^{\pm\eb_{d}}$ such that:
\begin{enumerate}
\item The process $\cb_{ij}$ is stationary and ergodic (with respect to $\muh$),
and dominates a supercritical Bernoulli percolation\footnote{Domination here
means that there exists a supercritical Bernoulli percolation process
on $\zz_L^d$ whose set of open edges is contained in the set of open edges given by $\cb$.}
 uniformly in $q$.
\item $T\le CL^{\lambda}$.
\item For any $\eb\in \{\pm\eb_{1},\dots,\eb_{d}\}$ the move $M^{\eb}$ satisfies:
\begin{enumerate}
\item The domain of $M^{\eb}$, $\dom M^{\eb}$, consists of the configurations
in which $\cb_{0,\eb}=1.$
\item $2^{\loss(M^{\eb})}\le C\,L^{\lambda}$.
\item For any $\eta\in\text{Dom}M^{\eb}$, denoting $M^{\eb}=\left\{ \eta_t(\eta),x^{t}(\eta),e^{t}(\eta)\right\} _{t=0}^{T}$,
at the final configuration
\[
\eta_{T}(\eta)=\eta^{0,\eb}.
\]
 
\end{enumerate}
\end{enumerate}
\end{lem}

\begin{proof}
The lemma is proven in \cite{ES20KAtagged}, Lemmas 3.9 and 3.14.
See also \cite[Section 3.4.1]{MST19KA}.
The reader may note that in the proof of  \cite{ES20KAtagged} $q$ is assumed small, but
since the relevant probabilities estimated are monotone in $q$
one may discard this assumption by adjusting the constants $C,\lambda$
in equation \eqref{eq:lowerboundL}.
\end{proof}
\begin{rem}\label{rem:droplet_picture}
The reason for the iterated exponential scaling of $D(\rho)$ hides
in the proof of Lemma \ref{lem:multistepmove}, and explained in details
in \cite{ES20KAtagged,MST19KA,ToninelliBiroliFisher}. It is based
on induction over both $k$ and $d$, of two different scales.

The first scale, $l(k,d)$, is the scale at which cluster of empty sites
could typically advance. For $k=1$, for example, the constraint is
always satisfied and $l(1,d)=1$. Perhaps more interesting is the
case $k=d=2$, where a column of empty sites of length $l$ could
move if there is an empty site in a neighboring column (see Figure
\ref{fig:droplet_propagation}). The probability to have a vacancy
in the neighboring column is $1 - (1-q)^l$, hence this event
becomes likely when choosing $l(2,2)\approx1/q$.
This is the scale of the \emph{droplets},
which are those empty clusters of size $l$ that are able to move
in $\zz^{d}$.

\begin{figure}
\begin{tikzpicture}[scale=0.4, every node/.style={scale=0.8}]
    \def \x{0}
    \def \y{0}	
	\draw[step=1, gray, very thin, shift={(\x,\y)}] (0,0) grid +(2,4);

	\foreach \yy in {0,1,2,3}{
		\draw[shift={(\x,\y)}] (0.5,\yy+0.5) node[black] {$0$};
	}
	\draw[shift={(\x,\y)}]  (1.5,0.5) node[black] {$0$};
		
	\draw[->,shift={(\x,\y)}]  (3,2) to (5,2);
		
    \def \x{6}
    \def \y{0}	
	\draw[step=1, gray, very thin, shift={(\x,\y)}] (0,0) grid +(2,4);

	\foreach \yy in {0,1,2,3}{
		\draw[shift={(\x,\y)}] (0.5,\yy+0.5) node[black] {$0$};
	}
	
	\draw[shift={(\x,\y)}]  (1.5,1.5) node[black] {$0$};
		
	\draw[->,shift={(\x,\y)}]  (3,2) to (5,2);
	
	\def \x{12}
    \def \y{0}	
	\draw[step=1, gray, very thin, shift={(\x,\y)}] (0,0) grid +(2,4);

	\foreach \yy in {1,2,3}{
		\draw[shift={(\x,\y)}] (0.5,\yy+0.5) node[black] {$0$};
	}
	
	\draw[shift={(\x,\y)}]  (1.5,0.5) node[black] {$0$};	
	\draw[shift={(\x,\y)}]  (1.5,1.5) node[black] {$0$};
		
	\draw[->,shift={(\x,\y)}]  (3,2) to (5,2);
	
	\def \x{18}
    \def \y{0}	
	\draw[step=1, gray, very thin, shift={(\x,\y)}] (0,0) grid +(2,4);

	\foreach \yy in {0,1,2}{
		\draw[shift={(\x,\y)}] (1.5,\yy+0.5) node[black] {$0$};
	}
	
	\draw[shift={(\x,\y)}]  (0.5,2.5) node[black] {$0$};	
	\draw[shift={(\x,\y)}]  (0.5,3.5) node[black] {$0$};
		
	\draw[->,shift={(\x,\y)}]  (3,2) to (5,2);
	
	\def \x{24}
    \def \y{0}	
	\draw[step=1, gray, very thin, shift={(\x,\y)}] (0,0) grid +(2,4);

	\foreach \yy in {0,1,2,3}{
		\draw[shift={(\x,\y)}] (1.5,\yy+0.5) node[black] {$0$};
	}
	
	\draw[shift={(\x,\y)}]  (0.5,2.5) node[black] {$0$};	
		
%
%
%
%
%
%
%
		
\end{tikzpicture}
\caption{\label{fig:droplet_propagation}Droplet propagation. Sites marked with $0$ are empty, the other sites could be either empty or occupied. We see that in a sequence of \emph{unconstrained} transitions the empty column moves to the right.}
\end{figure}

The second scale, $L(k,d)$, is the typical distance of an arbitrary site to a droplet,
so $L(k,d)\approx q^{-l(k,d)}$. If we look at a particle and consider
its neighborhood at scale $L(k,d)$, we are likely to find a droplet,
that would be able to move to the vicinity of that particle and help
it jump.

In order to understand the scaling of $D(\rho)$, we should understand
the two scales $l(k,d)$ and $L(k,d)$. Consider the set $[1,L(k-1,d-1)]^{d}$.
If we empty the entire boundary of this set, it could serve as a droplet
-- take, for example, the surface $\{0\}\times[1,L(k-1,d-1)]^{d-1}$.
This is a $d-1$ dimensional surface, and each of its sites has an
empty neighbor to the right coming from $[1,L(k-1,d-1)]^{d}$. Therefore,
any move for the KA dynamics with parameters $k-1,d-1$ could be applied
to that surface. Since its size is $L(k-1,d-1)$, it is likely to
contain a droplet. Hence, using this droplet, we are able to move
freely the sites on the surface. With slightly more careful analysis,
it could be shown that by rearranging the sites on $\{0\}\times[1,L(k-1,d-1)]^{d-1}$
the set $[1,L(k-1,d-1)]^{d}$ could ``swallow'' this surface, thus
moving one step to the left. That is, $[1,L(k-1,d-1)]^{d}$ is, indeed,
a droplet; and so $l(k,d)\approx L(k-1,d-1)$.

The two relations, $L(k,d)\approx q^{-l(k,d)}$ and $l(k,d)\approx L(k-1,d-1)$,
show that the two scales indeed behave as an iterated exponential.
The scaling of the diffusion coefficient could then be explained heuristically,
if we imagine that the particles are mostly blocked, except those
in the vicinity of a droplet. Since the sites that are able to move
have density $L^{-d}$, the diffusion coefficient scales polynomially
in $L$.
\end{rem}

An immediate consequence of point one of Lemma \ref{lem:multistepmove}
is that the graph induced by the open edges (i.e., for which $\cb$ equals $1$)
has a unique infinite connected component. Let $\mathcal{C}$ denote
this infinite component. In \cite{Faggionato08} (see also \cite{Spohn90SelfDiff}),
it is shown that the diffusion coefficient of a random walk on $\mathcal{C}$
is given by the following variational formula:
\[
\overline{D}=\inf_{\psi}\sum_{\alpha}\muh\left[\cb_{0,e_{\alpha}}\left(\delta_{\alpha,1}+\psi(\tau_{\eb_{\alpha}}\cb)-\psi(\cb)\right)^{2}|0\in\mathcal{C},\eb_{\alpha}\in\mathcal{C}\right],
\]
where the infimum is taken over local functions $\psi:\Pi\rightarrow\rr$
(namely, functions that depend on finitely many edges).

The input we need from \cite{Faggionato08,DeMasiFerrariGoldsteinWick}
is the positivity of the diffusion coefficient:
\begin{lem}
\label{lem:RWoncluster}There exists $\overline{D}_{0}>0$ such that
for all local $\psi:\Pi\to\rr$ and all $\rho\in(0,1)$, 
\[
\sum_{\alpha}\muh\left[\cb_{0,e_{\alpha}}\left(\delta_{\alpha,1}+\psi(\tau_{\eb_{\alpha}}\cb)-\psi(\cb)\right)^{2}\right]\ge\overline{D}_{0}.
\]
\end{lem}

\begin{proof}
This is a direct consequence of \cite[Lemma 2.1]{Faggionato08} and
the first point of Lemma \ref{lem:multistepmove}.
\end{proof}
In order to relate the diffusion coefficient given in equation (\ref{eq:vy})
to $\overline{D}$, we use the following proposition:
\begin{prop}
Fix a local function $g:\overline{\Omega}_L\times\Pi\rightarrow\rr$.
Then there exists a local function $\psi:\Pi\rightarrow\rr$, such
that 
\begin{multline*}
\sum_{\alpha=1}^{d}\muh\left[\cb_{0,\eb_{\alpha}}\left(\delta_{\alpha,1}+\psi(\tau_{\eb_{\alpha}}\cb)-\psi(\cb)\right)^{2}\right]\le\frac{1}{2\rho(1-\rho)}\times\\
\sum_{\alpha=1}^{d}\mub\otimes\muh\left[\cb_{0,\eb_{\alpha}}\left(\delta_{\alpha,1}(\etab(\eb_{1})-\etab(0))-\sum_{i\in\zz_{L}^{d}}\overline{\grad}_{0,\eb_{\alpha}}\,g(\tau_{i}\etab,\tau_{i}\cb)\right)^{2}\right],
\end{multline*}
where $\overline{\grad}$ is the gradient operating only on $\etab$
(that is, $\overline{\grad}_{0,\eb_{\alpha}}\,g(\tau_{i}\etab,\tau_{i}\cb)=g(\tau_{i}\,\etab^{0,\eb_{\alpha}},\tau_{i}\cb)-g(\tau_{i}\etab,\tau_{i}\cb)$).
\end{prop}

\begin{proof}
Note first that the sum $\sum_{i\in\zz_{L}^{d}}\overline{\grad}_{0,\eb_{\alpha}}\,g(\tau_{i}\etab,\tau_{i}\cb)$
is finite (and hence well defined) since $g$ is local. We are therefore
allowed, throughout the proof, to replace it by a sum over a large
enough torus $\mathbb{T}_{N,L}^{d}=\zz_{L}^{d}/N\zz_{L}^{d}$ for large
$N$ (depending on $g$). We start by writing the left hand side of
the inequality as 
\[
\sum_{\alpha=1}^{d}\muh\left[\cb_{0,\eb_{\alpha}}\left(\text{I}+\text{II}+\text{III}\right)\right],
\]
\begin{align*}
\text{I} & =\delta_{\alpha,1},\\
\text{II} & =2\delta_{\alpha,1}\left(\psi(\tau_{\eb_{1}}\cb)-\psi(\cb)\right),\\
\text{III} & =\left(\psi(\tau_{\eb_{\alpha}}\cb)-\psi(\cb)\right)^{2};
\end{align*}
and the right hand side (noting that $\cb$ depends only on $\etah$
and not on $\etab$) as 
\[
\sum_{\alpha=1}^{d}\muh\left[\cb_{0,\eb_{\alpha}}(\text{I}'+\text{II}'+\text{III}')\right],
\]
\begin{align*}
\text{I}' & =\mub\left[\delta_{\alpha,1}(\etab(\eb_{1})-\etab(0))^{2}\right],\\
\text{II}' & =-2\delta_{\alpha,1}\,\mub\left[(\etab(\eb_{1})-\etab(0))\,\sum_{i\in\mathbb{T}_{N,L}^{d}}\overline{\grad}_{0,\eb_{1}}\,g(\tau_{i}\etab,\tau_{i}\cb)\right],\\
\text{III}' & =\mub\left[\left(\sum_{i\in\mathbb{T}_{N,L}^{d}}\overline{\grad}_{0,\eb_{\alpha}}\,g(\tau_{i}\etab,\tau_{i}\cb)\right)^{2}\right].
\end{align*}

We now compare term by term. The term $\text{I},\text{I}'$ do not
depend on $\psi$: $\text{I}'=\delta_{\alpha,1}2\rho(1-\rho),$ so
indeed $\text{I}\le\frac{1}{2\rho(1-\rho)}\text{I}'$.

For the other terms we need to specify our choice of $\psi$:
\[
\psi(\cb)=2\mub\left[\etab(0)\sum_{i\in\mathbb{T}_{N,L}^{d}}g(\tau_{i}\etab,\tau_{i}\cb)\right].
\]

Fix $\eb\in\left\{ \eb_{1},\dots,\eb_{d}\right\} $, and note that $\tau_i \cb$
depends only on $\etah$ for any $i\in\zz^d_L$.
Then 
\begin{align*}
\psi(\tau_{\eb}\cb) & =2\mub\left[\etab(0)\sum_{i\in\mathbb{T}_{N,L}^{d}}g(\tau_{i}\etab,\tau_{i+\eb}\cb)\right]
=2\mub\left[\etab(\eb)\sum_{i\in\mathbb{T}_{N,L}^{d}}g(\tau_{i+\eb}\etab,\tau_{i+\eb}\cb)\right]\\
 & =2\mub\left[\etab(0)\sum_{i\in\mathbb{T}_{N,L}^{d}}g(\tau_{i}\,\etab^{0,\eb},\tau_{i}\cb)\right],
\end{align*}
and thus 
\begin{equation}
\psi(\tau_{\eb_{\alpha}}\cb)-\psi(\cb)
=\mub\left[2\etab(0)\sum_{i\in\mathbb{T}_{N,L}^{d}}\overline{\grad}_{0,\eb_{\alpha}}\,g(\tau_{i}\,\etab,\tau_{i}\cb)\right].\label{eq:translatios_to_grad}
\end{equation}
Observe now that $\etab(0)=\etab(\eb_{1})$ implies $\overline{\grad}_{0,\eb_{1}}\,g(\tau_{i}\,\etab,\tau_{i}\cb)=0$,
and otherwise $\etab(\eb_{1})=1-\etab(0)$, yielding 
\begin{align*}
(\etab(\eb_{1})-\etab(0))\overline{\grad}_{0,\eb_{1}}\,g(\tau_{i}\,\etab,\tau_{i}\cb) & =(1-2\etab(0))\overline{\grad}_{0,\eb_{1}}\,g(\tau_{i}\,\etab,\tau_{i}\cb).
\end{align*}
Therefore, by equation \eqref{eq:translatios_to_grad}, 
\[
\mub\left[(\etab(\eb_{1})-\etab(0))\sum_{i\in\mathbb{T}_{N,L}^{d}}\overline{\grad}_{0,\eb_{1}}\,g(\tau_{i}\,\etab,\tau_{i}\cb)\right]=\sum_{i\in\mathbb{T}_{N,L}^{d}}\mub\left[\overline{\grad}_{0,\eb_{1}}\,g(\tau_{i}\,\etab,\tau_{i}\cb)\right]-(\psi(\tau_{\eb_{1}}\cb)-\psi(\cb)),
\]
and noting that $\mub\left[\overline{\grad}_{0,\eb_{1}}\,g(\tau_{i}\,\etab,\tau_{i}\cb)\right]=0$
(the gradient of any function has $0$ expected value), we obtain
\[
\text{II}=\text{II}'.
\]
Finally, for the last term we use again equation (\ref{eq:translatios_to_grad}),
together with Jensen's inequality and the fact that $\etab(0)^{2}\le1$:
\[
\text{III}\le\mub\left[\left(2\etab(0)\sum_{i\in\mathbb{T}_{N,L}^{d}}\overline{\grad}_{0,\eb_{\alpha}}\,g(\tau_{i}\,\etab,\tau_{i}\cb)\right)^{2}\right]\le4\,\text{III}'. \qedhere
\]
\end{proof}
\begin{cor}
\label{cor:coarsegrained_annealed}For all local $g:\overline{\Omega}_L\times\Pi\rightarrow\rr$,
\[
\frac{1}{2\rho(1-\rho)}\sum_{\alpha=1}^{d}\mub\otimes\muh\left[\cb_{0,\eb_{\alpha}}\left(\delta_{\alpha,1}(\etab(\eb_{1})-\etab(0))-\sum_{i\in\zz_{L}^{d}}\overline{\grad}_{0,\eb_{\alpha}}\,g(\tau_{i}\etab,\tau_{i}\cb)\right)^{2}\right]\ge\overline{D}_{0},
\]
where $\overline{D}_{0}$ is the positive constant given in Lemma
\ref{lem:RWoncluster}.
\end{cor}

The next step of the proof is to use the multistep move given in Lemma
\ref{lem:multistepmove} in order to compare $\overline{D}_{0}$ with
$D^{(0)}$ (recalling equation \eqref{eq:vy}).
\begin{prop}
\label{prop:pathargument}Fix a local function $f:\Omega\rightarrow\rr$.
Then there exists a local function $g:\overline{\Omega}_L\times\Pi\rightarrow\rr$
such that 
\begin{multline*}
\mu\left(\sum_{\alpha=1}^{d}c_{0,e_{\alpha}}(\eta)\left(\delta_{\alpha,1}(\eta(e_{\alpha})-\eta(0))-\sum_{x\in\zz^{d}}\grad_{0,e_{\alpha}}(\tau_{x}f)\right)^{2}\right)\ge\\
L^{-\lambda}\,\sum_{\alpha=1}^{d}\mub\otimes\muh\left[\cb_{0,\eb_{\alpha}}\left(\delta_{\alpha,1}(\etab(\eb_{1})-\etab(0))-\sum_{i\in\zz_{L}^{d}}\overline{\grad}_{0,\eb_{\alpha}}\,g(\tau_{i}\etab,\tau_{i}\cb)\right)^{2}\right].
\end{multline*}
\end{prop}

\begin{proof}
Let $g(\etab,\cb)=\mu\left[\frac{1}{L}\sum_{y\in\left[L\right]^{d}}\tau_{y}f(\eta)\middle|\etab,\cb\right]$.
We use Lemma \ref{lem:multistepmove} in order to write, for all $x\in\zz^{d}$
and $\alpha\in\left\{ 1,\dots,d\right\} $, denoting 
$M^{\eb_{\alpha}}=\left\{ \eta_{t}(\eta),x^{t}(\eta),e^{t}(\eta)\right\} _{t=0}^{T}$,
\begin{equation}
\overline{\grad}_{0,\eb_{\alpha}}\tau_{x}f=\sum_{t=1}^{T}\grad_{x^{t},x^{t}+e^{t}}\tau_{x}f(\eta_{t})=\sum_{t=1}^{T}\tau_{x^{t}}\,\grad_{0,e^{t}}\tau_{x-x^{t}}f(\eta_{t}).
\end{equation}
We also note that the total particle flow (defined as the change in $\sum_{x}x\eta(x)$)
can be decomposed along the $T$-step move. In more details, using the fact that
$\eta_0$ and $\eta_T$ agree outside $\{0,\eb_\alpha\}$,
\[
\sum_x x (\eta_0(x) - \eta_T(x)) = L e_\alpha \left(\etab(\eb_\alpha)-\etab(0)\right).
\]
On the other hand, at step $t$ the configuration changes only at $x^t$ and $x^t+e^t$, therefore
\[
\sum_x x (\eta_{t-1}(x) - \eta_{t}(x))
     = e^t (\eta_{t}(x^t) - \eta_{t}(x^t+e^t))
     = e^t \tau_{x^t} (\eta_{t}(0) - \eta_{t}(e^t)),
\]
implying
\[
\sum_x x (\eta_0(x) - \eta_T(x)) 
	= \sum_{t=1}^T \sum_x x (\eta_{t-1}(x) - \eta_{t}(x))
	= \sum_{t=1}^T e^t \tau_{x^t} (\eta_{t}(0) - \eta_{t}(e^t)).
\]
That is,
\begin{equation}
Le_{\alpha}\left(\etab(\eb_{\alpha})-\etab(0)\right)
	= \sum_{t=1}^{T}e^{t}\tau_{x^{t}}\left(\eta^{t}(e^{t})-\eta_{t}(0)\right).
\end{equation}

Using these two identities, the Cauchy-Schwarz inequality, and the
properties of the move, we obtain 
\begin{multline*}
\sum_{\alpha=1}^{d}\mub\otimes\muh\left[\cb_{0,\eb_{\alpha}}\left(e_{1}\cdot e_{\alpha}(\etab(\eb_{\alpha})-\etab(0))-\sum_{i\in\zz_{L}^{d}}\overline{\grad}_{0,\eb_{\alpha}}\,g(\tau_{i}\etab,\tau_{i}\cb)\right)^{2}\right] \\
\le
\frac{1}{L^{2}}\sum_{\alpha=1}^{d}\mu\left[\cb_{0,\eb_{\alpha}}\left(e_{1}\cdot Le_{\alpha}(\etab(\eb_{1})-\etab(0))-\sum_{i\in\zz_{L}^{d}}\overline{\grad}_{0,\eb_{\alpha}}\,\sum_{y\in\left[L\right]^{d}}\tau_{i+y}f(\eta)\right)^{2}\right]\\
=
\frac{1}{L^{2}}\sum_{\alpha=1}^{d}\mu\left[\cb_{0,\eb_{\alpha}}\left(e_{1}\cdot\sum_{t=1}^{T}e^{t}\tau_{x^{t}}\left(\eta_{t}(e^{t})-\eta_{t}(0)\right)-\sum_{x\in\zz^{d}}\sum_{t=1}^{T}\tau_{x^{t}}\,\grad_{0,e^{t}}\tau_{x-x^{t}}f(\eta_{t})\right)^{2}\right]\\
\le
\frac{T}{L^{2}}\sum_{t=1}^{T}\sum_{\alpha=1}^{d}\mu\left[\cb_{0,\eb_{\alpha}}\tau_{x^{t}}\,c_{0,e^{t}}(\eta_{t})\left(e_{1}\cdot e^{t}\left(\eta_t(e^t)-\eta_t(0)\right)-\sum_{x\in\zz^{d}}\grad_{0,e^t}\tau_x f(\eta_t)\right)^{2}\right]\\
\le
\frac{T}{L^{2}}\sum_{t=1}^{T}\sum_{\alpha=1}^{d}\sum_{\eta\in\Omega}\mu(\eta)\sum_{\eta'\in\Omega}\One_{\eta'=\eta_{t}}\sum_{\beta=1}^{d}\One_{e_{\beta}=e^{t}}c_{0,e_{\beta}}(\eta')\left(e_{1}\cdot e_{\beta}\left(\eta'(e_{\beta})-\eta'(0)\right)-\sum_{x\in\zz^{d}}\grad_{0,e_{\beta}}\tau_{x}f(\eta')\right)^{2} \\ 
\le
\frac{T^{2}}{L^{2}}\sum_{\alpha=1}^{d}2^{\text{Loss}(M^{\eb_{\alpha}})}\sum_{\eta'\in\Omega}\mu(\eta')\sum_{\beta=1}^{d}c_{0,e_{\beta}}(\eta')\left(e_{1}\cdot e_{\beta}\left(\eta'(e_{\beta})-\eta'(0)\right)-\sum_{x\in\zz^{d}}\grad_{0,e_{\beta}}\tau_{x}f(\eta')\right)^{2}.
\end{multline*}
The result follows by inserting the bounds for $T$ and $\loss(M)$
given in Lemma \ref{lem:multistepmove}.
\end{proof}
The proof of the lower bound (\ref{eq:lowerbound}) follows from Proposition
\ref{prop:pathargument}, Corollary \ref{cor:coarsegrained_annealed},
and the variational characterization of $D^{(0)}$ in equation (\ref{eq:vy}).\qed

\section{Proof of the upper bound}

In order to find the upper bound we will use a process tightly
related to the Kob-Andersen model, called the \emph{$k$-neighbor
bootstrap percolation} (see, e.g., \cite{Morris17BP}).
We start by defining this process, and describing some of its
basic properties.

\subsection{Bootstrap percolation}

\begin{defn}[bootstrap percolation]
Fix $V\subseteq\zz^{d}$ and $A\subseteq \zz^d$. The\emph{ bootstrap
percolation in $V$ starting from $A$} is a deterministic process
defined for $t=0,1,2,\dots$ as 
\begin{align*}
A_{0} & =A\cap V,\\
A_{t+1} & =A_{t}\cup\{x\in V:\#\{y\in A_{t}\text{ such that }y\sim x\}\ge k\}.
\end{align*}
The limit $\cup_{t\ge0}A_{t}$ is called the \emph{span of $A$ in
$V$}, and denoted by $[A]^{V}$. We say that two sites $x$ and $y$
are \emph{connected for the bootstrap percolation in $V$ starting
from $A$} if they are connected in $[A]^{V}$ (thought of as the
subgraph of $\zz^{d}$ induced by the set $[A]^{V}$), that is, if
there is a nearest neighbor path $x=x_{1},\dots,x_{n}=y$ such that
$x_{1},\dots x_{n}\in[A]^{V}$.

For $\eta\in\Omega$, we define 
\[
A_{\eta}=\{x \in \zz^{d}:\eta_{x}=0\}.
\]
We may refer to the bootstrap percolation in $V$ starting from $A_{\eta}$
as the bootstrap percolation starting from $\eta$. When context allows
we omit the explicit mention of $V$, $A$, or both.
\end{defn}

We continue with several properties of bootstrap percolation.
\begin{observation}
\label{obs:bp_monotonicity}(\emph{monotonicity}). Let $U\subseteq V\subseteq\zz^{d}$, and
fix $A\subseteq B \subseteq\zz^{d}$. Then $[A]^{U}\subseteq[A]^{V}$ and $[A]^{U}\subseteq[B]^{U}$.
\end{observation}
The following observation reveals the the connection between bootstrap percolation and
the Kob-Andersen model:
\begin{observation}
\label{obs:BPandKA}Fix $\eta\in\Omega$, and consider a set $V\subset\zz^{d}$.
Assume that, for two neighboring sites $x,y\in V$, the constraint
$c_{x,y}$ is satisfied in $V$, that is, $c_{x,y}(\eta')=1$ for any $\eta'$
that agrees with $\eta$ on $V$.
Then $[A_{\eta}]^{V}=[A_{\eta^{x,y}}]^{V}$.
\end{observation}

\begin{proof}
Assume without loss of generality that $\eta(x)=1$ and $\eta(y)=0$,
and note that $[A_{\eta}]^{V}\subseteq[A_{\eta}\cup\{x\}]^{V}$. On
the other hand, since $c_{x,y}=1$ in $V$, the site $x$ will be
added to $A_{\eta}$ after a single step of the bootstrap percolation.
Denoting the set after that single step by $A'$, $[A_{\eta}\cup\{x\}]^{V}\subseteq[A']^{V}=[A_{\eta}]^{V}$.
Therefore $[A_{\eta}]^{V}=[A_{\eta}\cup\{x\}]^{V}$. The same argument
shows that $[A_{\eta^{x,y}}]^{V}=[A_{\eta}\cup\{x\}]^{V}$.
\end{proof}

\begin{observation}
\label{obs:bpcluster}Fix $A\subset\zz^{d}$, $V\subset\zz^{d}$,
and $x\in V$. Let $U$ be the set of sites connected to $x$ in $[A]^{V}$.
Then $[A]^{U}=U$.
\end{observation}

\begin{proof}
Let $(A_{t})_{t\ge0}$ denote the bootstrap percolation in $V$ starting
with $A$, and assume by contradiction $[A]^{U}\subsetneq U$. Since
$U\subseteq[A]^{V}$, there exists a first time $t$ for which some
$y\in U\setminus[A]^{U}$ is contained in $A_{t}$. By minimality,
$A_{t-1}\cap U\subseteq[A]^{U}$, and since $y\notin[A]^{U}$ it has
at most $k-1$ neighbors in $A_{t-1}\cap U$. On the other hand, it
has at least $k$ neighbors in $A_{t}$. Therefore, it must have at
least one neighbor in $V\setminus U$. This is a contradiction, since
$U$ is a connected component containing $y$.
\end{proof}
\begin{claim}
\label{claim:bp_distantinfluence}Fix $A\subset\zz^{d}$. Consider
two sets $B\subset B'\subset\zz^{d}$, a site $z\in B$, and any $S\subset\zz^{d}$.
Assume that $z$ is connected to $S$ for the bootstrap percolation
in $B'$, but not for the bootstrap percolation in $B$. Then $z$
is connected to $\partial B$ for the bootstrap percolation in $B'$.
\end{claim}

\begin{proof}
Assume that $z$ is not connected to $\partial B$ for the bootstrap
percolation in $B'$, so in particular its connected component in
$[A]^{B'}$, denoted $U$, is entirely contained in $B$. By Observation
\ref{obs:bpcluster} and monotonicity of the bootstrap percolation,
$U=[A]^{U}\subseteq[A]^{B}$. This is a contradiction, since by assumption
$U\cap S\neq\emptyset$ but $[A]^{B}\cap S=\emptyset$.
\end{proof}

\subsection{Analysis of the test function}
We will prove the upper bound by estimating the expression inside the infimum
in equation \eqref{eq:vy} for a carefully chosen function $f$.

Recall $q=1-\rho$. The test function we will construct will depend on a scale 
\begin{equation}
l=\floor{ \exp^{k-2}(\lambda q^{-\frac{1}{d-k+1}}) }.\label{eq:lupper}
\end{equation}
Throughout the section $\lambda$ and $C$ denote generic positive
constants.
\begin{defn}[relevant sites]
Fix $\eta\in\Omega$. A site $x\in\left[-2l,2l\right]^{d}$ is called
\emph{relevant} if it is not connected to $\left\{ 0,1\right\} \times\left[-2l,2l\right]^{d-1}$
for the bootstrap percolation in $\left[-2l,2l\right]^{d}$; and otherwise
it is called \emph{irrelevant}. Denote the set of relevant sites by
$\mathcal{R}(\eta)$.
\end{defn}

We divide the box $\left[-l,l\right]^{d}$ in two parts -- the left
part $\Lambda_{-}=\left[-l,0\right]\times\left[-l,l\right]^{d-1}$,
and the right part $\Lambda_{+}=\left[1,l\right]\times\left[-l,l\right]^{d-1}$ (see Figure \ref{fig:upper_bound_geometry}).
The test function we consider is 
\begin{equation}
f(\eta)=\frac{1}{2\left(2l+1\right)^{d-1}}\left(\sum_{x\in\Lambda_{+}\cap\mathcal{R}}\eta(x)-\sum_{x\in\Lambda_{-}\cap\mathcal{R}}\eta(x)\right).\label{eq:testfunction}
\end{equation}
Hence, the purpose of this section is to prove that for $\varepsilon$
small enough
\[
\mu\left[\sum_{\alpha=1}^{d}c_{0,e_{\alpha}}^{(\varepsilon)}\left(\delta_{\alpha,1}\left(\eta(e_{1})-\eta(0)\right)-\sum_{x\in\zz^{d}}\grad_{0,e_{\alpha}}\tau_{x}f\right)^{2}\right]\le e^{-\lambda l}.
\]

\begin{rem}
The choice of $f$ in equation \eqref{eq:testfunction} seems mysterious at first sight --
Observation \ref{obs:BPandKA} explains the use of bootstrap percolation, but the
introduction of relevant sites and the exact form of $f$ are not that clear.

One way to gain more intuition on this choice of $f$ is to look more carefully at the
variational principle \eqref{eq:vy}. Ignoring the contribution of $f$, we are left with the term
\[
c_{0,e_\alpha} \left(\delta_{\alpha,1} \left(\eta(e_1) - \eta(0) \right)\right)^2.
\]
This could be though of as a contribution of the instantaneous current between the origin and $e_1$.
The appearance of this term is not surprising -- if typically the system has large currents, it is
natural to expect the diffusion coefficient to be large.

However, the typical instantaneous current is not sufficient to understand the behavior
of the diffusion coefficient -- correlations in space and time could also have an
important effect. For example, take the Kob-Andersen model with $k=d=2$, and consider
a configuration in which the sites $e_1$,$e_2$ and $e_1+e_2$ are empty, and all other sites
(in a large neighborhood of the origin) are occupied. The particle at the origin could jump
one step to the right, but any attempt to jump further is not allowed by the constraint.
Therefore, if we wait for some time it is likely to jump back to the left. Thus, we see that
an instantaneous right current can cause at a later time a current to the left.
The role of the function $f$ in equation \eqref{eq:vy} is to compensate for this effect,
by adding to $\delta_{\alpha,1} \left(\eta(e_1) - \eta(0) \right)$ an effective current in
the opposite direction.

The example of the last paragraph demonstrates the following heuristic picture --
typically, most particles are confined to a very small region; they can move back and forth
but never too far. Assume for simplicity that the origin is occupied, and consider the particle
there. In view of the heuristic described above, this particle will remain for a very long
time in a certain region that we may refer to as the \emph{attainable region}. Recalling
Observation \ref{obs:BPandKA}, it is reasonable to approximate this attainable region by
the set of sites connected to the origin for the bootstrap percolation in some (large) box.
Hence, being \emph{relevant} roughly represents a small attainable region.

As long is this attainable region remains small, we expect that any instantaneous current
to the right caused by the particle at the origin will be canceled shortly after by a jump
to the left.
For a good choice of $f$, this fact (assuming $c_{0,e_1}(\eta)=1$) should be expressed as 
\[
\eta(e_1) - \eta(0) \approx \sum_{x\in\zz^d}\grad_{0,e_1} \tau_x f(\eta).
\]

In the following we will see that the function $f$ defined in equation \eqref{eq:testfunction}
satisfies this approximated relation. The error term corresponds to the possibility that
the attainable region is, in fact, large. When the notion of "small" or "large" attainable region
is determined according to the scale $l$ given in equation \eqref{eq:lupper}, we obtain the
upper bound of $D$.
%

\end{rem}

First, observe that since $f$ depends on $(4l+1)^{d}$ sites and
its maximum is smaller than $l+1$,
\[
\mu\left[\sum_{\alpha=1}^{d}\varepsilon\left(\delta_{\alpha,1}\left(\eta(e_{1})-\eta(0)\right)-\sum_{x\in\zz^{d}}\grad_{0,e_{\alpha}}\tau_{x}f\right)^{2}\right]
\le d\varepsilon(1+(4l+1)^{d}(l+1))^{2}=O(\varepsilon).
\]
Therefore, since $c_{0,e_{\alpha}}^{(\varepsilon)}=(1-\varepsilon)c_{0,e_{\alpha}}+\varepsilon$,
it suffices to prove 
\begin{equation}
\mu\left[\sum_{\alpha=1}^{d}c_{0,e_{\alpha}}\left(\delta_{\alpha,1}\left(\eta(e_{1})-\eta(0)\right)-\sum_{x\in\zz^{d}}\grad_{0,e_{\alpha}}\tau_{x}f\right)^{2}\right]\le e^{-\lambda l}.\label{eq:testineq}
\end{equation}

Since the analysis of $f$ will require us to understand when particles
enter or exit different boxes (and in particular $\Lambda_{\pm}$),
we will need to introduce some notation. First, for a set $\Lambda\subset\zz^{d}$,
we say that an (undirected) edge $(x,y)$ is on the \emph{boundary}
of $\Lambda$, and write $(x,y)\in\overline{\partial}\Lambda$, if
one vertex is in $\Lambda$ and the other outside $\Lambda$. The
\emph{(inner) boundary} $\partial\Lambda$ are the sites in $\Lambda$
that have a neighbor outside $\Lambda$.

For $\alpha=1,\dots,d$ we define the \emph{boundary in the $e_{\alpha}$
direction} 
\[
\partial^{\alpha}\Lambda=\left\{ x:(x,x-e_{\alpha})\in\overline{\partial}\Lambda\right\} .
\]
We will write $\Lambda_{l}=\left[-l,l\right]^{d}$ (and $\Lambda_{2l}=\left[-2l,2l\right]^{d}$),
as well as 
\[
\Lambda_{l}^{\alpha}=\left[-l,l\right]^{\alpha-1}\times\left\{ 0\right\} \times\left[-l,l\right]^{d-\alpha}.
\]
Finally, for $x_{0}\in\Lambda_{l}^{\alpha}$, we denote the two boundary
sites above and below $x_{0}$ as 
\begin{align*}
x_{0}^{+\alpha} & =x_{0}+\left(l+1\right)e_{\alpha},\\
x_{0}^{-\alpha} & =x_{0}-le_{\alpha}.
\end{align*}
Note that $x_{0}^{\pm\alpha}\in\partial^{\alpha}\Lambda_{l}$. See Figure \ref{fig:upper_bound_geometry}.

\begin{figure}
\begin{tikzpicture}[scale=0.5, every node/.style={scale=1.0}]
	\def \l{5}
	
	\draw[OliveGreen, very thick] (0,-\l) -- (0,\l);	
	
	\foreach \x in {-\l,...,0}{
		\foreach \y in {-\l,...,\l}{
			\fill[BrickRed] (\x,\y) circle (0.1);
		}
	}
	
	\foreach \x in {1,...,\l}{
		\foreach \y in {-\l,...,\l}{
			\fill[NavyBlue] (\x,\y) circle (0.1);
		}
	}
	
	\foreach \i in {-\l,...,\l}{
		\draw[gray] (\i,\l) -- +(0,1);
		\draw[gray] (\i,-\l) -- +(0,-1);
		\draw[gray] (\l,\i) -- +(1,0);
		\draw[gray] (-\l,\i) -- +(-1,0);
	}
	
	\foreach \i in {-\l,...,\l}{
		\draw[gray] (\i,\l) circle (0.25);
		\draw[gray] (\i,-\l) circle (0.25);
		\draw[gray] (\l,\i) circle (0.25);
		\draw[gray] (-\l,\i) circle (0.25);
	}
	
	\foreach \i in {-\l,...,\l}{
		\draw[YellowOrange] (\l+1,\i) circle (0.4);
		\draw[YellowOrange] (-\l,\i) circle (0.4);
	}

\end{tikzpicture}
\caption{\label{fig:upper_bound_geometry} A sketch of the different sets in $\Lambda_l$ in two dimensions. The left part $\Lambda_-$ are the red vertices; the right part $\Lambda_+$ consists of the blue vertices; the boundary $\overline{\partial}\Lambda$ is given by the gray edges; the inner boundary $\partial \Lambda$ contains the sites circled in gray; the boundary in the $e_1$ direction $\partial^1 \Lambda$ is circled in orange; and the set $\Lambda_l^1$ is represented by the green line crossing the box. In this picture, for any $x_0 \in \Lambda_l^1$ (i.e., on the green line), the site $x_0^{+1}$ is the orange circle to its right (at distance $l+1$); and the site $x_0^{-1}$ is the orange circle to its left (at distance $l$).}
\end{figure}

We are now ready to start the analysis of $f$. In the next proposition we will see,
for fixed $x$, what is the contribution of $\grad_{0,e_1} \tau_x f$:
\begin{prop}
\label{prop:non0grad}Fix an edge $\left(x,x-e\right)$ and configuration
$\eta$ such that $c_{0,e}=1$ and $\grad_{0,e}\tau_{x}f\neq0$. Then
one of the following holds:
\begin{enumerate}
\item $0\in x+\left(\Lambda_{2l+1}\setminus\Lambda_{2l-2}\right)$ (equivalently
$x\in\Lambda_{2l+1}\setminus\Lambda_{2l-2}$), and there exists $y\in x+\partial\Lambda_{l}$
such that the bootstrap percolation in $x+\Lambda_{2l}$ connects
$y$ to $x+\partial\Lambda_{2l-2}$, either for $\eta$ or $\eta^{0,e}$.
In this case $\left|\grad_{0,e}\tau_{x}f\right|\le C\,l$. See Figure \ref{fig:non0grad}.
\item $\left(0,e\right)\in x+\overline{\partial}\Lambda_{l}$ (equivalently
$\left(x,x-e\right)\in\overline{\partial}\Lambda_{l}$) and $-x$
is relevant for $\tau_{x}\eta$. In this case 
\begin{align}
\grad_{0,e}\tau_{x}f & =\frac{\eta(e)-\eta(0)}{2(2l+1)^{d-1}}\times\begin{cases}
1 & 0\in x+\Lambda_{+}\cap\partial\left[-l,l\right]^{d}\\
-1 & e\in x+\Lambda_{+}\cap\partial\left[-l,l\right]^{d}\\
-1 & 0\in x+\Lambda_{-}\cap\partial\left[-l,l\right]^{d}\\
1 & e\in x+\Lambda_{-}\cap\partial\left[-l,l\right]^{d}
\end{cases}\nonumber \\
 & =\frac{\eta(e)-\eta(0)}{2(2l+1)^{d-1}}\times\begin{cases}
1 & x\in\Lambda_{-}\cap\partial\left[-l,l\right]^{d}\\
-1 & x-e\in\Lambda_{-}\cap\partial\left[-l,l\right]^{d}\\
-1 & x\in\Lambda_{+}\cap\partial\left[-l,l\right]^{d}\\
1 & x-e\in\Lambda_{+}\cap\partial\left[-l,l\right]^{d}
\end{cases}.\label{eq:testfunction_grad}
\end{align}
\end{enumerate}
\end{prop}

\begin{figure}
\begin{tikzpicture}[scale=0.12, every node/.style={scale=0.6}]
	\def \l{10}

	\draw[step=1, gray, ultra thin, opacity=0.5] (-2*\l-1.5,-2*\l-1.5) grid (2*\l+1.5,2*\l+1.5);	
	
	\filldraw[black]  (0,0) circle (4pt) node[anchor=west] {$x$};
	
	\draw[BrickRed, thick, fill=gray, fill opacity=0.3] (-2*\l,-2*\l)--(-2*\l,2*\l)--(2*\l,2*\l)--(2*\l,-2*\l)--(-2*\l,-2*\l);
	\draw[BrickRed, thick] (-2*\l-1,-2*\l-1)--(-2*\l-1,2*\l+1)--(2*\l+1,2*\l+1)--(2*\l+1,-2*\l-1)--(-2*\l-1,-2*\l-1);
	\draw[BrickRed, thick] (-2*\l+1,-2*\l+1)--(-2*\l+1,2*\l-1)--(2*\l-1,2*\l-1)--(2*\l-1,-2*\l+1)--(-2*\l+1,-2*\l+1);
	
	\draw[NavyBlue, thick] (-\l,-\l)--(-\l,\l)--(\l,\l)--(\l,-\l)--(-\l,-\l);
	
	\draw[OliveGreen, thick] (-2*\l+2,-2*\l+2)--(-2*\l+2,2*\l-2)--(2*\l-2,2*\l-2)--(2*\l-2,-2*\l+2)--(-2*\l+2,-2*\l+2);
	
	\draw[thick] (\l,0.2*\l) .. controls (1.3*\l,-0.3*\l) and (1.7*\l,0.7*\l) .. (2*\l-2,0.5*\l);
\end{tikzpicture}
\caption{\label{fig:non0grad} Illustration of the first case in Proposition \ref{prop:non0grad}. $x+\Lambda_{2l}$ is filled in gray. The origin is on one of the three red lines, which represent $x + \Lambda_{2l+1}\setminus \Lambda_{2l-2}$. The blue square is $x + \partial \Lambda_l$ and the green square is $x + \partial \Lambda_{2l-2}$. These two squares must be connected for the bootstrap percolation.}
\end{figure}

\begin{proof}
$f$ could only change when the set of relevant sites changes, or
when a relevant site changes its occupation.

The first case corresponds to point 1 -- for the set of relevant
sites for $\tau_{x}\eta$ to change, $[A_{\eta}]^{x+\Lambda_{2l}}$
must change, and by Observation \ref{obs:BPandKA} this is only possible
if $c_{0,e}$ is only satisfied with the help of sites outside $x+\Lambda_{2l}$.
This means that at least one of the vertices $0$ or $e$ is in
$x + \partial \Lambda_{2l}$, and in particular
$0\in x+\left(\Lambda_{2l+1}\setminus\Lambda_{2l-2}\right)$.

To understand the second implication, we may assume without loss of
generality that there is some site $z\in\Lambda_{l}$ which is connected
to $\{0,1\}\times[-2l,2l]^{d-1}$ in $[A_{\eta}]^{x+\Lambda_{2l}}$
but not in $[A_{\eta^{0,e}}]^{x+\Lambda_{2l}}$. By monotonicity of
bootstrap percolation and using again Observation \ref{obs:BPandKA},
$z$ cannot be connected to $\{0,1\}\times[-2l,2l]^{d-1}$ in $[A_{\eta}]^{\Lambda_{2l-2}}$.
Then, by Claim \ref{claim:bp_distantinfluence}, $z$ is connected
to $\partial\Lambda_{2l-2}$ in $[A_{\eta}]^{x+\Lambda_{2l}}$. To
finish the first point, we only need the rough bound $|f(\eta)|\le\frac{|\Lambda_{+}|+|\Lambda_{-}|}{2(2l+1)^{d-1}}$.

In the second case, we note first if a particle jumps inside $\Lambda_l$,
and it is originally in $\Lambda_+ \cap \mathcal{R}$, then it will remain
in $\Lambda_+ \cap \mathcal{R}$ (and analogously for $\Lambda_-$). Therefore,
$f$ could only change if a particle jumps into or out of $\Lambda_l$, so
for $\tau_x f$ to change we must require $\left(0,e\right)\in x+\overline{\partial}\Lambda_l$.
Moreover, its (shifted) position $-x$ must be relevant for the shifted
configuration $\tau_x \eta$.
Then, $\grad_{(0,e)}\tau_{x}f$ is given by following
carefully the four options: moving into $\Lambda_{+}$, out of
$\Lambda_{+}$, into $\Lambda_{-}$, or out of $\Lambda_{-}$.
\end{proof}

In order to control the contribution coming from the first case
of Proposition \ref{prop:non0grad}, we give a name to the event
that appears there:
\begin{defn}
Fix an edge $(x,x-e)$. Then $E_{x,x-e}$ is the event, that there
exist $y\in x+\partial\Lambda_{l}$ and $z\in x+\partial\Lambda_{2l-2}$
such that the bootstrap percolation in $x+\Lambda_{2l}$ connects $y$
to $z$, either for $\eta$ or $\eta^{0,e}$.
\end{defn}

An important tool we will use in order to bound the probability of
this event is the following lemma:
\begin{lem}[{\cite[Lemma 5.1]{CerfManzo}}]
\label{lem:CM}Let $l'<10l$, and fix $y,z\in\Lambda_{l'}$. Then,
assuming that the constant $\lambda$ in equation (\ref{eq:lupper})
is small enough, 
\begin{align*}
\mu(z\text{ connected to }y\text{ in }[A_{\eta}]^{\Lambda_{l'}}) & \le\left(C\norm{z-y}_{\infty}^{d-1}q\right)^{\lambda\norm{z-y}_{\infty}} & \quad k=2,\\
\mu(z\text{ connected to }y\text{ in }[A_{\eta}]^{\Lambda_{l'}}) & \le q^{\lambda\norm{z-y}_{\infty}} & \quad k\ge3.
\end{align*}
\end{lem}

\begin{claim}
\label{claim:becomingrelevant_prob}Fix an edge $(x,x-e)$. Then 
\[
\mu(E_{x,x-e})\le Ce^{-\lambda l}.
\]
\end{claim}

\begin{proof}
First, note that there are $Cl^{d-1}$ possible choices of $y$
and $Cl^{d-1}$ choices of $z$. Note that $\norm{z-y}_\infty \ge l-2$.
By Lemma \ref{lem:CM}, the probability for $y$
to be connected to $z$ is bounded by $(Cl^{d-1}q)^{\lambda l}$ for
$k=2$ and $q^{\lambda l}$ for $k\ge3$; both of which are, indeed,
smaller than $Ce^{-\lambda l}$.
\end{proof}
The last claim covers the first case of Proposition \ref{prop:non0grad},
and we now move to the second. 

When considering this case,
we will use certain cancellations in directions perpendicular to $e_1$.
More precisely, a particle jumping from $0$ to $e_\alpha$ (for $\alpha\neq 1$)
will \emph{enter} $x + \Lambda_l$ (for an appropriate choice of $x$) but \emph{exit}
$x' + \Lambda_l$ (for an appropriate choice of $x'$). We then expect
that $\grad_{0,e_\alpha} \tau_x f$ and $\grad_{0,e_\alpha} \tau_{x'} f$
will cancel out, which is indeed the case unless one of the sites is relevant
and the other irrelevant. We therefore introduce the following event:
\begin{defn}
Fix $\alpha\in\{2,\dots,d\}$ and $x_{0}\in\Lambda_{l}^{\alpha}$.
Let $E_{\alpha}(x_{0})$ be the event,
that either $-x_{0}^{+\alpha}$ is relevant for $\tau_{x_{0}^{+\alpha}}\eta$,
or $-x_{0}^{-\alpha}$ is relevant for $\tau_{x_{0}^{-\alpha}}\eta$,
but not both.
\end{defn}

\begin{claim}
\label{claim:updown_nocanlelation}Fix $\alpha\in\{2,\dots,d\}$ and
$x_{0}\in\Lambda_{l}^{\alpha}$.
Then for all $\eta\in E_{\alpha}(x_{0})$, the origin is connected
to $\partial\Lambda_{l}$ in $[A_{\eta}]^{\Lambda_{3l}}$. Moreover,
\[
\mu(E_{\alpha}(x_{0}))\le Ce^{-\lambda l}.
\]
\end{claim}

\begin{proof}
We will prove that the origin is connected
to $\partial\Lambda_{l}$ in $[A_{\eta}]^{\Lambda_{3l}}$
for the case where
$-x_{0}^{+\alpha}$ is relevant for $\tau_{x_{0}^{+\alpha}}\eta$, but
$-x_{0}^{-\alpha}$ is irrelevant for $\tau_{x_{0}^{-\alpha}}\eta$.
The complementing case as analogous.

Let $S=x_{0}+\left\{ 0,1\right\} \times\zz^{d-1}$,$B_{-}=x_{0}^{-\alpha}+\Lambda_{2l}$,$B_{+}=x_{0}^{+\alpha}+\Lambda_{2l}$.
Saying that $-x_{0}^{+\alpha}$ is relevant for $\tau_{x_{0}^{+\alpha}}\eta$
is the same as saying that $0$ is connected to $B_{+}\cap S$ in
$[A_{\eta}]^{B_{+}}$; and saying that $-x_{0}^{-\alpha}$ is irrelevant
for $\tau_{x_{0}^{-\alpha}}\eta$ is the same as saying that that
$0$ is not connected to $B_{-}\cap S$ in $[A_{\eta}]^{B_{-}}$.

In particular, setting $z=0$, $B=B_{-}$ and $B'=\Lambda_{3l}$,
$A_{\eta}$ satisfies the conditions of Claim \ref{claim:bp_distantinfluence}.
Therefore $0$ is connected to $\partial B_{-}$ in $[A_{\eta}]^{B'}$,
which implies the result since $0\in\Lambda_{l}\subset B_{-}$.

The probability estimate follows from Lemma \ref{lem:CM}.
\end{proof}
\begin{claim}
\label{claim:updowncancelation}Fix $\alpha\in \{2,\dots,d \}$, and a configuration
$\eta$ such that $\eta\notin\bigcup_{x_{0}\in\Lambda_{l}^{\alpha}}E_{\alpha}(x_{0})$
and $c_{0,e_{\alpha}}(\eta)=1$. Then 
\[
\sum_{x\in\partial^{\alpha}\Lambda_{l}}\grad_{0,e_{\alpha}}\tau_{x}f=0.
\]
\end{claim}

\begin{proof}
We split the sum according to the projection of $x$ on $\Lambda_{l}^{\alpha}$
-- 
\[
\sum_{x\in\partial^{\alpha}\Lambda_{l}}\grad_{0,e_{\alpha}}\tau_{x}f=\sum_{x_{0}\in\Lambda_{l}^{\alpha}}\left(\grad_{0,e_{\alpha}}\tau_{x_{0}^{+\alpha}}f+\grad_{0,e_{\alpha}}\tau_{x_{0}^{-\alpha}}f\right).
\]
Fix one of these summands. If $-x_{0}^{+\alpha}$ is irrelevant for
$\tau_{x_{0}^{+\alpha}}\eta$ and 
$-x_{0}^{-\alpha}$ is irrelevant for $\tau_{x_{0}^{-\alpha}}\eta$,
then by the Proposition \ref{prop:non0grad}
\[
\grad_{0,e_\alpha} \tau_{x_0^{+\alpha}} f = \grad_{0,e_\alpha} \tau_{x_0^{-\alpha}} f = 0.
\]
Otherwise, since $\eta\notin E_{\alpha}(x_{0})$, both must be relevant, hence
\begin{align*}
\grad_{0,e_{\alpha}}\tau_{x_{0}^{+\alpha}}f & =\frac{\eta(e_{\alpha})-\eta(0)}{2\left(2l+1\right)^{d-1}}\times\begin{cases}
-1 & x_{0}\in\Lambda_{-},\\
1 & x_{0}\in\Lambda_{+};
\end{cases}\\
\grad_{0,e_{\alpha}}\tau_{x_{0}^{-\alpha}}f & =\frac{\eta(e_{\alpha})-\eta(0)}{2\left(2l+1\right)^{d-1}}\times\begin{cases}
1 & x_{0}\in\Lambda_{-},\\
-1 & x_{0}\in\Lambda_{+};
\end{cases}
\end{align*}
and their sum is $0$.
\end{proof}
\begin{claim}
\label{claim:perpendicular_directions}Fix $\alpha \in \{2,\dots,d \}$. Then 
\[
\mu\left[c_{0,e_{\alpha}}\left(\sum_{x\in\zz^{d}}\grad_{0,e_{\alpha}}\tau_{x}f\right)^{2}\right]\le Ce^{-\lambda l}.
\]
\end{claim}

\begin{proof}
We split in the different cases described in Proposition \ref{prop:non0grad}:
\[
\mu\left[c_{0,e_{\alpha}}\left(\sum_{x\in\zz^{d}}\grad_{0,e_{\alpha}}\tau_{x}f\right)^{2}\right] \le 
2 \mu\left[c_{0,e_{\alpha}}\left(\sum_{x\in\Lambda_{2l+1}\setminus\Lambda_{2l-2}}\grad_{0,e_{\alpha}}\tau_{x}f\right)^{2}\right]+2\mu\left[c_{0,e_{\alpha}}\left(\sum_{x\in\partial^{\alpha}\Lambda_{l}}\grad_{0,e_{\alpha}}\tau_{x}f\right)^{2}\right].
\]
We can bound the first term using Claim \ref{claim:becomingrelevant_prob}:
\[
\mu\left[c_{0,e_{\alpha}}\left(\sum_{x\in\Lambda_{2l+1}\setminus\Lambda_{2l-2}}\One_{E(x,x-e_{\alpha})}Cl\right)^{2}\right]\le Cl^{d}\mu\left[\sum_{x}\One_{E(x,x-e_{\alpha})}\right]\le Ce^{-\lambda l}.
\]
The second term, according to Claim \ref{claim:updowncancelation},
vanishes unless $\eta \in E_{\alpha}(x_{0})$ for some $x_0\in \Lambda_l^\alpha$,
so we are left with an error term which by Claim \ref{claim:updown_nocanlelation}
is bounded by
\[
\mu\left[\left(\frac{|\partial^{\alpha}\Lambda_{l}|}{2(2l+1)^{d-1}}\right)^{2}\sum_{x_{0}\in\Lambda_{l}^{\alpha}}\One_{E_{\alpha}(x_{0})}\right]\le Ce^{-\lambda l}. \qedhere
\]
\end{proof}

The next step is to consider the direction $e_1$:
\begin{claim} \label{claim:prob_irrelevant}
Fix $x\in\partial^{1}\left[-l,l\right]^{d}$. Then $-x$ is irrelevant
for $\tau_{x}\eta$ with probability smaller than $Ce^{-\lambda l}$.
\end{claim}

\begin{proof}
For $-x$ to be irrelevant it must be connected to one of $2(4l+1)^{d-1}$
sites on $\{0,1\}\times[-2l,2l]^{d-1}$. All of these sites are at
distance at least $l-2$ from $x$, and the statement follows by direct
application of Lemma \ref{lem:CM}.
\end{proof}

\begin{claim}
\label{claim:paralleldirection}For $e=e_{1}$, 
\[
\mu\left[c_{0,e}\left(\eta(e)-\eta(0)-\sum_{x\in\zz^{d}}\grad_{0,e}\tau_{x}f\right)^{2}\right]\le Ce^{-\lambda l}.
\]
\end{claim}

\begin{proof}
The proof of the claim consists in showing that each site on $\partial^{1}\Lambda_{l}$
contributes $\frac{\eta(e)-\eta(0)}{|\partial^{1}\Lambda_{l}|}$ to
the sum, up to a small error term.

First, using Proposition \ref{prop:non0grad}, we write 
\begin{align*}
\mu\left[c_{0,e}\left(\eta(e)-\eta(0)-\sum_{x\in\zz^{d}}\grad_{0,e}\tau_{x}f\right)^{2}\right] \le &
	2\mu\left[c_{0,e}\left(\sum_{x\in\Lambda_{2l+1}\setminus\Lambda_{2l-2}}\grad_{0,e}\tau_{x}f\right)^{2}\right] \\
 & +2\mu\left[c_{0,e}\left(\eta(e)-\eta(0)-\sum_{x\in\partial^{1}\Lambda_{l}}\grad_{0,e}\tau_{x}f\right)^{2}\right].
\end{align*}
The first term, just as in the proof of Claim \ref{claim:perpendicular_directions},
is bounded by $Ce^{-\lambda l}$ according to Claim \ref{claim:becomingrelevant_prob}.

In order to bound the second term, we start by assuming that all sites
of $-\partial^{1}\Lambda_{l}$ are relevant. In this case, 
\[
\sum_{x\in\partial^{1}\Lambda_{l}}\grad_{0,e}\tau_{x}f=\sum_{x\in\partial^{1}\Lambda_{l}}\frac{\eta(e)-\eta(0)}{2(2l+1)^{d-1}}=\eta(e)-\eta(0),
\]
so 
\[
\mu\left[c_{0,e}\left(\eta(e)-\eta(0)-\sum_{x\in\partial^{1}\Lambda_{l}}\grad_{0,e}\tau_{x}f\right)^{2}\One_{-\partial^{1}\Lambda_{l}\subseteq\mathcal{R}}\right]=0.
\]
Finally, by Claim \ref{claim:prob_irrelevant}
the probability that $\partial^{1}\Lambda_{l}$
contains irrelevant sites is smaller than $Ce^{-\lambda l}$: 
\[
\mu\left[c_{0,e}\left(\eta(e)-\eta(0)-\sum_{x\in\partial^{1}\Lambda_{l}}\grad_{0,e}\tau_{x}f\right)^{2}\One_{-\partial^{1}\Lambda_{l}\not\subseteq\mathcal{R}}\right]\le Ce^{-\lambda l}.
\]
The claim thus follows by summing the contribution of the three terms.
\end{proof}
All that is left is to combine Claims \ref{claim:perpendicular_directions}
and \ref{claim:paralleldirection}, proving inequality (\ref{eq:testineq})
and hence the second part of Theorem \ref{thm:main}. \qed

\section{Further problems}
\begin{itemize}
\item Prove convergence to a hydrodynamic limit without the soft constraint
from a more restricted family of initial states (as in \cite{GoncalvesLandimToninelli}).
\item Improve the bounds on the diffusion coefficient, and in particular
find matching upper and lower bound without a logarithmic correction.
In the case of the closely related Fredrickson-Andersen model, where
similar bounds have been obtained for the spectral gap (\cite{TowardsUniversality}),
the logarithmic correction could be removed, and, moreover, the exact
constant multiplying $1/(1-\rho)^{d-k+1}$ could be identified \cite{HMT2020}.
\item Study qualitative properties of the diffusion coefficient -- is it decreasing
in $\rho$? Is it continuous? Smooth? Is $D^{(0)}=D$?
\item Understand the hydrodynamic limit of more KCLGs. The comparison argument
of Section \ref{sec:lowerbound} could be used in order to estimate
the diffusion coefficient whenever an appropriate multistep move
could be constructed, and may be useful in lager generality than presented
here.
\item The bounds on the diffusion coefficient may have consequences other
than the hydrodynamic limit -- in general, we expect the correlation
$\mu(\eta(0)e^{t\mathcal{L}}\eta(x))-\rho^{2}$ to behave like $\rho(1-\rho)(4\pi t\,D)^{-d/2}\,e^{-\frac{x^{2}}{4tD}}$
(see, e.g., \cite{Spohn2012IPS}). It has been shown in \cite{CMRT2010}
that in the Kob-Andersen model, for $x=0$, this correlation decays at
least as fast as $C\,(\log t)^{5}/t$ for some unidentified constant $C$.
Any progress towards the predicted $\rho(1-\rho)(4\pi t\,D)^{-d/2}\,e^{-\frac{x^{2}}{4tD}}$
(for $D$ is as in equation \eqref{eq:D_eps0}) would be an interesting result.
\end{itemize}

\section*{Acknowledgments}

I would like to thank Cl\'ement Erignoux, Alessandra Faggionato,
Fabio Martinelli, and Patr\'icia Gon\c{c}alves for very useful discussions.
I acknowledge the support of the ERC Starting Grant 680275 MALIG.

\appendix

\section{\label{sec:appendix}The gradient condition in cooperative models}

In this appendix we will see that cooperative kinetically constrained
lattice gas models (KCLGs) are non-gradient.

A general KCLG is a Markov process with configuration space $\Omega=\left\{ 0,1\right\} ^{\zz^{d}}$,
determined by a set of constraints giving each edge $(x,y)\in\mathcal{E}(\zz^{d})$
a rate $c_{x,y}(\eta)\in\{0\}\cup[1,\infty)$, for any configuration
$\eta\in\Omega$. We will make the following assumptions:
\begin{enumerate}
\item The model is homogeneous, i.e., the constraint is translation invariant.
\item The constraint $c_{x,y}$ depends only on the configuration outside
$x$ and $y$.
\item The constraints have finite range, i.e., $c_{x,y}$ depends only on
the occupation of sites in the box $x+\Lambda_{R}$, where $R$ is
called the \emph{range}.
\item The constraint is non-degenerate, i.e., for every edge $(x,y)$ of
$\zz^{d}$ there exist a configuration $\eta$ such that $c_{x,y}(\eta)>0$
and $\eta'$ such that $c_{x,y}(\eta')=0$.
\item For fixed $x,y$ the constraint $c_{x,y}(\eta)$ is a decreasing function
of $\eta$, i.e., adding more empty sites could only help the constraint
to be satisfied.
\end{enumerate}
With such constraints, the process is given by a generator as in equation
(\ref{eq:generator}).
\begin{defn}[connected configurations]
Fix a KCLG and two configurations $\eta,\eta'$. We say that $\eta'$
is \emph{connected} to $\eta$ if there exists a sequence of configuration
$\eta_{0},\dots,\eta_{T}$ such that $\eta_{0}=\eta$, $\eta_{T}=\eta'$,
and for all $t\in\{0,\dots,T-1\}$ there exist $x_{t+1}\sim y_{t+1}$
such that $\eta_{t+1}=\eta_{t}^{x_{t+1},y_{t+1}}$, with $c_{x_{t+1},y_{t+1}}(\eta_{t})\ge1$.
For any fixed $e\in \{\pm e_1,\dots,\pm e_d\}$, we say that $\eta'$ is $e$\emph{-connected}
to $\eta$ if, in addition, $y_{t+1}=x_{t+1}+e$ and $\eta_{t}(x_{t})=0$,
namely, all transitions move a vacancy in the direction $e$ (or,
equivalently, a particle in the direction $-e$). Note that $\eta'$
is connected to $\eta$ if and only if $\eta$ is connected to $\eta'$;
and $\eta'$ is $e$-connected to $\eta$ if and only if $\eta$ is
$(-e)$-connected to $\eta'$.
\end{defn}

\begin{defn}
Let $A\subseteq\zz^{d}$. The configuration $\eta_{A}$ is defined
as 
\[
\eta_{A}(x)=\begin{cases}
0 & x\in A,\\
1 & \text{otherwise}.
\end{cases}
\]
\end{defn}

KCLGs could be either \emph{cooperative} or \emph{non-cooperative}
(see \cite[Definition 1.1]{CMRT2010}). We remind here that a non-cooperative
model is a model in which there exists a \emph{mobile cluster}, defined
as follows:
\begin{defn}[mobile cluster]\label{def:mobile_cluster}
Let $A$ be a finite non-empty subset of $\zz^{d}$. We say that
$A$ is a \emph{mobile cluster} if:
\begin{enumerate}
\item For all $z\in\zz^{d}$, the configuration $\eta_{A}$ is connected
to the configuration $\eta_{z+A}$. 
\item For every edge $(x,y)$, there exists a translation $z\in\zz^{d}$
such that $c_{x,y}(\eta_{z+A})\ge1$.
\end{enumerate}
\end{defn}
%

Gradient models, in our context, are interacting particle systems
with conserved number of particles, in which the current is a gradient
of some local function. This property significantly
simplifies the analysis of their hydrodynamic limits (see, e.g., \cite[Definition 2.5]{KipnisLandim}).
The purpose of this appendix is to prove the following result:
\begin{thm}
\label{thm:cooperativeisnongradient}Cooperative KCLGs are non-gradient.
\end{thm}

In order to prove that a model is non-gradient, we will consider the
model on a torus, and show that the integral of the current does not
always vanish:
\begin{lem}
\label{fact:gradient0integral}Consider a KCLG, and assume that for
$N$ large enough, there exists a configuration on the torus $\eta\in\left\{ 0,1\right\} ^{\mathbb{T}^d_N}$,
such that 
\[
\sum_{x,y\in\mathbb{T}^d_N}(x-y)(\eta(x)-\eta(y))c_{x,y}(\eta)\neq0.
\]
Then the model is non-gradient.
\end{lem}

\begin{proof}
Assume that the model is gradient. That is, by \cite[Definition 2.5]{KipnisLandim},
for some $n_0 \in \nn$, for any $1\le i \le d, 1 \le n \le n_0$ there exist
a cylinder function $h_{i,n}$ and a finite range function $p_{i,n}$ satisfying
$\sum_{x\in \mathbb{T}^d_N}p_{i,n}(x)=0$,
such that the current is given by
\[
W_{0,e_i} (\eta) = \sum_{n=1}^{n_0} \sum_{x \in \mathbb{T}^d_N} p_{i,n}(x)\tau_x h_{i,n}(\eta)
\] 
for every $i$.

Then for any $e\in \{e_1,\dots, e_d\}$
\begin{align*}
\sum_{y\in\mathbb{T}^d_N} W_{y,y+e}(\eta)
&= \sum_{y\in\mathbb{T}^d_N}
   \tau_{-y} \sum_{n=1}^{n_0} \sum_{x \in \mathbb{T}^d_N} p_{i,n}(x)\tau_x h_{i,n}(\eta) \\
&= \sum_{x \in \mathbb{T}^d_N} \sum_{z\in\mathbb{T}^d_N}
   \sum_{n=1}^{n_0} p_{i,n}(x)\tau_z h_{i,n}(\eta)\\
&= 0.
\end{align*}
This concludes the proof, recalling
\[
W_{x,x+e}(\eta) = (\eta(x)-\eta(x+e))c_{x,y}(\eta). \qedhere
\]
\end{proof}

The construction of such $\eta$ for a cooperative KCLG is based on
the notion of reachable sites:
\begin{defn}[reachable sites and $e$-stretch]
We say that a site is \emph{reachable} from a configuration $\eta$
if it is empty for some $\eta'$ which is connected to $\eta$. For
$e\in\left\{ \pm e_{1},\dots,\pm e_{d}\right\} $ we say that a site
is $e$-reachable for a configuration $\eta$ if it is empty for some
$\eta'$ which is $e$-connected to $\eta$. The \emph{$e$-stretch}
of $\eta$ is defined as 
\[
\sup\left\{ e\cdot x:x\text{ is }e\text{-reachable}\right\} .
\]
\end{defn}

By the definition of non-cooperative models, it is immediate that
if $\eta$ contains a mobile cluster then for every site $x$ there
exists $\eta'$ connected to $\eta$ for which $\eta'(x)=0$. In the
next proposition we will see that if we require $e$-connectivity
the converse is also true --
\begin{prop}
\label{prop:infinitereachable_noncooperative}Assume that for all
$e\in\left\{ \pm e_{1},\dots,\pm e_{d}\right\} $ there exists a finite
subset $A_{e}$ of $\zz^{d}$, such that the $e$-stretch of $\eta_{A_{e}}$
is infinite. Then the model is non-cooperative.
\end{prop}

Before proving this proposition, we will see how it implies Theorem
\ref{thm:cooperativeisnongradient}. Consider a cooperative KCLG,
so by Proposition \ref{prop:infinitereachable_noncooperative} for
some $e\in\{\pm e_{1},\dots,\pm e_{d}\}$ and any $L\in\nn$, configurations
that are entirely filled outside $\Lambda_{L}$ have finite $e$-stretch.
We will assume without loss of generality that $e=e_{1}$.

Since the model is non-degenerate, there exists a configuration $\eta_{0}$
for which $c_{0,e_{1}}(\eta_{0})=1$. Since the model has finite range
$R$, we may assume that this configuration is entirely filled outside
$\Lambda_{R}$; and since the constraint does not depend on the occupation
at $0$ and $e_{1}$ we assume $\eta_{0}(0)=0$ and $\eta_{0}(e_{1})=1$.
We will now construct a sequence of configuration starting at $\eta_{0}$,
so that $\eta_{i+1}$ is obtained from $\eta_{i}$ by moving a $0$
to the right, i.e., $\eta_{i+1}=\eta_{i}^{x_{i},x_{i}+e_{1}}$ for
some $x_{i}$ such that $c_{x_{i},x_{i}+e_{1}}(\eta_{i})>0$, $\eta_{i}(x_{i})=0$,
and $\eta_{i}(x_{i}+e_{1})=1$. When, for some $i$, more than one
such choice of $x$ is possible, we choose one arbitrarily. We stop
when none of the sites satisfy the required conditions.

Since the $e_{1}$-stretch of $\eta_0$ is finite the construction must stop at
some step $n<\infty$. On the other hand, we chose $\eta_{0}$ such
that $c_{0,e_{1}}(\eta_{0})\ge1$, $\eta_{0}(0)=0$, and $\eta_{0}(e_{1})=1$,
so $n\ge1$. Hence, for the configuration $\eta=\eta_{n}$, for all
$x\in\zz^{d}$
\[
c_{x,x+e_{1}}(\eta)(1-\eta(x))\eta(x+e_{1})=0,
\]
but for $x^{*}=x_{n-1}$ (using $c_{x^{*},x^{*}+e_{1}}(\eta_n)=c_{x^{*},x^{*}+e_{1}}(\eta_{n-1})=1$), we know that
\[
c_{x^{*},x^{*}+e_{1}}(\eta)\eta(x^{*})(1-\eta(x^{*}+e_{1}))\ge1.
\]
That is, 
\[
\sum_{x\in\zz^{d}}(\eta(x)-\eta(x+e_{1}))c_{x,x+e_{1}}(\eta)\ge1.
\]
Since $\eta$ is filled outside $\Lambda_{R+n}$, we may as well sum
over $x$ in a large enough torus $\mathbb{T}^d_{100R+n}$. Therefore,
by Lemma \ref{fact:gradient0integral}, the model is indeed non-gradient.
\qed \\

We return to the proof of Proposition \ref{prop:infinitereachable_noncooperative}. 
\begin{claim}
\label{claim:infinitestretch_movingcluster}Fix a finite non-empty
$A\subset\zz^{d}$, and $e\in\left\{ \pm e_{1},\dots,\pm e_{d}\right\} $.
Assume that the $e$-stretch of $\eta_{A}$ is infinite. Then there
exists a finite non-empty $A'\subset\zz^{d}$ and a strictly positive
integer $n$, such that $\eta_{A'}$ is $e$-connected to $\eta_{ne+A'}$.
\end{claim}

\begin{proof}
First, we may assume without loss of generality that $A$ has the
minimal possible size, among sets for which the $e$-stretch of $\eta_{A}$
is infinite; and for notational convenience we also assume $e=e_{1}$.
Set $k=\left|A\right|$, and fix $L$ such that $A\subset\Lambda_{L}$.

We will start by showing the following property:
\begin{claim}
For all $j<k$, there exists $s^{(j)}$ such that for all $B\subset(-\infty,0]\times\zz^{d-1}$
with $\left|B\right|=j$, the $e_{1}$-stretch of $\eta_{B}$ is at
most $s^{(j)}$. In particular, there exists $L^{(j)}$ such that
the maximal possible $e_{1}$-stretch for such a set is obtained for
some \textbf{$B\subset[-L^{(j)},0]\times\zz^{d-1}$}.
\end{claim}

\begin{proof}
For $j=1$ choosing $s^{(1)}=L^{(1)}=0$ suffices since no particle
could move. For $j>1$, let $L^{(j)}=j(h^{(j-1)}+R)$ and $s^{(j)}$
the maximal $e_{1}$-stretch of $\eta_{B}$ for any \textbf{$B\subset[-L^{(j)},0]\times\zz^{d-1}$}.
Note that $s^{(j)}$ is well defined since particles cannot move in
directions orthogonal to $e_{1}$, so we may assume without loss of
generality that $B\subset[-L^{(j)},0]\times[-jR,jR]^{d-1}$; and it
is finite since $j<k$.

Assume now that for some $B\subset(-\infty,0]\times\zz^{d-1}$ of
size $j$ the $e_{1}$-stretch of $\eta_{B}$ is more than $s^{(j)}$.
We can assume without loss of generality that $0\in B$, and by construction
there must be a site $x\in B$ outside $[-L^{(j)},0]\times\zz^{d-1}$.
Due to our choice of $L^{(j)}$, the set $B$ could be separated by
a strip of width $h^{(j-1)}+R$, namely, there exists $n\in\zz$ such
that 
\begin{align*}
B & =B_{-}\cup B_{+},\\
B_{-} & \subset(-\infty,n]\times\zz^{d},\\
B_{+} & \subset(n+h^{(j-1)}+R,0]\times\zz^{d}.
\end{align*}
However, since the $e_{1}$-stretch of $\eta_{B_{-}}$ is at most
$h^{(j-1)}$, it would never be able to influence transitions to the
right of $n+h^{(j-1)}+R$, thus the $e_{1}$-stretch of $B$ cannot
be larger than that of $B_{-}$, which is a contradiction.
\end{proof}
As a result of this claim, there exists $s<\infty$, such that for
any set $B$ of size strictly less than $k$, the $e_{1}$-stretch
of $B$ is at most $s$ plus its maximal $e_{1}$ coordinate.

Since the $e_{1}$-stretch of $\eta_{A}$ is infinite, there exists
an $e_{1}$-reachable site $x$ with $e\cdot x>\binom{(2L+1)^{d-1}k(s+R)}{k}+s+1$.
Consider a sequence of \textbf{$T$} flips which empties that site.
We denote the set of empty sites at step $t$ by $A_{t}$, so that
$A_{0}=A$ and $A_{T}\ni x$; and $a_{t}$ denotes the rightmost coordinate
of $A_{t}$ (i.e., $a_{t}=\max_{y\in A_{t}}\{e_{1}\cdot y\}$). Assume
now that at some time $t$ we are able to identify a non-empty set
$\tilde{A}_{t}$ whose rightmost coordinate is $\tilde{a}_{t}$, such
that all sites of $A_{t}\setminus\tilde{A}_{t}$ are at least $s+R$
to the right of $\tilde{a}_{t}$, i.e., $a_{t}<e_{1}\cdot y-s-R$
for all $y\in A_{t}\setminus\tilde{A}_{t}$. We then know that the
$0$'s coming from $\tilde{A}_{t}$ will never be able to reach distance
$R$ from the sites of $A_{t}\setminus\tilde{A}_{t}$, thus the set
$A_{t}\setminus\tilde{A}_{t}$ moves as if these sites were filled.
In particular, it could not go further than distance $s$, hence $a_{t}>\binom{(2L+1)^{d-1}k(h+s)}{k}+1$.
That means that for at least $\binom{(2L+1)^{d-1}k(s+R)}{k}+1$ times
$t$ with different values of $a_{t}$, 
\[
A_{t}\subset\left[a_{t}-k(s+R),a_{t}\right]\times\left[-L,L\right]^{d-1}.
\]
This box has volume $(2L+1)^{d-1}k(s+R)$, so by the pigeonhole principle
there exist $t$ and $t'$ with $a_{t}<a_{t'}$ such that $A_{t}-a_{t}e_{1}=A_{t'}-a_{t'}e_{1}$.
This finishes the proof by taking $A'=A_{t}-a_{t}e_{1}$ and $n=a_{t'}-a_{t}$,
and using the translation invariance of the model.
\end{proof}
\begin{claim}
\label{claim:empyingtraslations}Fix any finite $B\subset\zz^{d}$
and $e\in\zz^{d}$, and assume that there exists a finite non-empty
$A\subset\zz^{d}$ such that the $e$-stretch of $\eta_{A}$ is infinite.
Then there exist a finite non-empty set $A'\subset\zz^{d}$ such that
for all $m\in\nn$, the configuration $\eta_{A'}$ is $e$-connected
to a configuration $\eta_{m}$ in which all the sites of $me+B$ are
empty. Moreover, we can assume that no site after $me+B$ is empty,
i.e., $\eta_{m}(x)=1$ whenever $x\cdot e>m+\sup_{y\in B}y\cdot e$.
\end{claim}

\begin{proof}
By the Claim \ref{claim:infinitestretch_movingcluster} there exists
$L\in\nn$, $A''\subset\Lambda_{L}$, and $n\in\nn$, such that $\eta_{A''}$
is $e$-connected to $\eta_{ne+A''}$. Note that we may, equivalently,
choose any $A''$ which is a translation of $A_{\eta}$ for any $\eta$
in the path connecting $\eta_{A''}$ with $\eta_{ne+A''}$. We will
therefore assume without loss of generality that $0\in A''$, but
$e\cdot x<0$ for all $x\in A\setminus\{0\}$.

Denote $B=\{b_{1},\dots,b_{k}\}$, with $e\cdot b_{1}\ge\dots\ge e\cdot b_{k}$,
and consider the union 
\[
A_{0}=\bigcup_{i=1}^{k}\left(b_{i}+A''-inLe\right).
\]
This union is disjoint, since $A''\subset\Lambda_{L}$, and by repeating
$L$ times the sequence of flips required to move $A''$ to $ne+A''$,
we can move $b_{1}+A''-nLe$ to $b_{1}+A''$, reaching a configuration
in which $b_{1}$ is empty. Then, repeating this sequence again $2L$
times we can move $b_{2}+A''-2nLe$ to $b_{2}+A''$. This is allowed
since during the first sequence we do not changes the configuration
at the sites of $b_{2}+A''-2nLe$; and the in the resulting configuration
both $b_{1}$ and $b_{2}$ are empty. We continue in the same manner,
until we reach a configuration $\eta_{0}'$ in which the sites of
$B$ are all empty.

Consider now for $j=0,\dots,n-1$ the set 
\[
A_{j}=A_{0}-knLje+je.
\]
As before, applying repeatedly the sequence that allowed us to move
$A''$ we can reach a configuration $\eta_{j}$ (connected to $\eta_{A_{j}}'$)
in which the sites of $je+B$ are empty. Furthermore, $A_{j}$ and
$A_{j'}$ are disjoint for $j\neq j'$, so, indeed, taking 
\[
A'=\bigcup_{j=0}^{n-1}A_{j},
\]
for $j=0,\dots,n-1$, the configuration $\eta_{A'}$ is $e$-connected
to a configuration $\eta_{j}$ for which the sites of $je+B$ are
empty. Finally, since $A'$ is a disjoint union of copies of $A''$,
we can translate each of them by $ne$, and if we do that in the right
order (starting with $b_{1}+A''-nLe$ and ending with $b_{k}+A''-knL(n-1)e+(n-1e)$)
they will never intersect. Hence $\eta_{ne+A'}$ is $e$-connected
to $\eta_{A'}$, and the result follows.
\end{proof}
\begin{claim}
\label{claim:flippingedgetotheright}Fix $e\in\left\{ \pm e_{1},\dots,\pm e_{d}\right\} $
and $L\in\nn$. Assume that there exists a finite non-empty $A\subset\zz^{d}$
such that the $e$-stretch of $\eta_{A}$ is infinite. Then there
exists $L'$ and $A'\subset\Lambda_{L'}$ such that for all $x\in\left[L',\infty\right]\times\left[-L,L\right]^{d-1}$
and every configuration $\eta$ for which the sites of $A'$ are empty,
$\eta$ is connected to $\eta^{x,x+e}$.
\end{claim}

\begin{proof}
We assume without loss of generality that $e=e_{1}$. The first observation
needed in order to prove this claim, is that there is a configuration
for which the constraint $c_{x,x+e_{1}}$ is satisfied, but none of
the sites to the right of $x$ are empty, i.e., $x+\left[1,\infty\right]\times\zz^{d-1}$
is entirely occupied. This is true since, if the $e_{1}$-stretch
of $\eta_{A}$ is infinite for finite $A$, at some point the rightmost
$0$ has to move to the right.

We then find a finite non-empty $B_{0}\subset\left[-\infty,0\right]\times\zz^{d-1}\setminus\{0\}$
such that $c_{0,e_{1}}(\eta_{B_{0}})=1$. Let 
\[
B=\bigcup_{z\in\{0\}\times\left[-L,L\right]^{d-1}}\left(z+B_{0}\right).
\]
Then, in particular, $c_{x,x+e_{1}}(\eta_{B})=0$ for $x\in\{0\}\times\left[-L,L\right]^{d-1}$.

We now apply Claim \ref{claim:empyingtraslations} to find a finite
non-empty set $A'\subset\zz^{d}$ such that for all $m\in\nn$, the
configuration $\eta_{A'}$ is $e$-connected to a configuration $\eta_{m}$
in which all the sites of $me+B$ are empty. We define $L'$ such
that $A'\subset\Lambda_{L'}$, and then, for every $x\in\left[L',\infty\right]\times\left[-L,L\right]^{d-1}$,
taking $m=e_{1}\cdot x$ yields $c_{x,x+e_{1}}(\eta_{m})=1$. Therefore,
if we take any configuration $\eta$ for which $A'$ is empty, by
performing the same transitions that connected $\eta_{A'}$ to $\eta_{m}$,
we reach a configuration for which $c_{x,x+e_{1}}=1$. Note that this
is done without changing the configuration neither at $x$ nor at
$x+e_{1}$. We then exchange $x$ and $x+e_{1}$, and fold back all
the transitions we have done before, reaching the configuration $\eta^{x,x+e_{1}}$.
\end{proof}
\begin{claim}
\label{claim:russiandoll}Assume that for all $e\in\{e_{1},\dots,e_{d}\}$
there exists a finite set $A_{e}\subset\zz^{d}$ such that the $e$-stretch
of $\eta_{A_{e}}$ is infinite, and fix $e'\in\{e_{1},\dots,e_{d}\}$.
Then there exists $L\in\nn$ and $A\subset\Lambda_{L}$ such that
for any $\eta$ in which the sites of $A$ are empty, and any $x\in\left[L+1,\infty\right]^{d}$,
the configuration $\eta^{x,x+e'}$ is connected to $\eta$.
\end{claim}

\begin{proof}
Without loss of generality we fix $e=e_{1}$. By Claim \ref{claim:flippingedgetotheright}
we can define $L_{1}\in\nn$ and $A_{1}\subset\Lambda_{L_{1}}$ be
such that for all $x_{1}\in\left[L_{1},\infty\right]\times\{0\}^{d-1}$
and every configuration $\eta$ for which the sites of $A_{1}$ are
empty, $\eta$ is connected to $\eta^{x_{1},x_{1}+e_{1}}$. Using
Claim \ref{claim:empyingtraslations} we can find $L_{2}\in\nn$ and
$A_{2}\in\Lambda_{L_{2}}$ such that, for every $x_{2}\in\{0\}\times\left[L_{2},\infty\right]\times\{0\}^{d-1}$,
the configuration $\eta_{A_{2}}$ is connected to a configuration
$\eta$ in which the sites of $x_{2}+A_{1}$ are empty, and during
the sequence of configurations connecting the two only edges of $\left[-\infty,-L_{2}\right]^{d}$
were flipped. We continue in the same manner, for $i=1,\dots,d$,
to construct $L_{i}$ and $A_{i}\subset\Lambda_{L_{i}}$ such that
for all $x_{i}\in\{0\}^{i-1}\times\left[L_{i},\infty\right]\times\{0\}^{d-i}$,
the configuration $\eta_{A_{i}}$ is connected to a configuration
in which the sites of $x_{i}+A_{i-1}$ are empty, and during the sequence
of configurations connecting the two only edges of $\left[-\infty,-L_{i}\right]^{d}$
were flipped.

Let $L=L_{d}$, $A=A_{d}$, and fix $\eta$ in which the sites of
$A$ are empty and $x\in\left[L+1,\infty\right]^{d}$. We write $x=x_{1}+\dots+x_{d}$
for $x_{i}\in\{0\}^{i-1}\times\left[L_{i},\infty\right]\times\{0\}^{d-i}$.
By our construction of $A$, $\eta$ is connected to a configuration
$\eta'$ in which the set $A_{1}+x_{2}+\dots+x_{d}$ is empty, and
during the sequence of configurations connecting the two the sites
$x$ and $x+e_{1}$ remained untouched. Then, by the construction
of $A_{1}$, we can connect $\eta'$ to $\eta'^{x,x+e_{1}}$. All
that is left is to rewind the steps leading to $\eta'$, and the proof
is complete.
\end{proof}
\begin{claim}
\label{claim:flipping_edges}Assume that for all $e\in\{e_{1},\dots,e_{d}\}$
there exists a finite set $A_{e}\subset\zz^{d}$ such that the $e$-stretch
of $\eta_{A_{e}}$ is infinite. Then there exists $L\in\nn$ and $A\subset\Lambda_{L}$
such that for any $\eta$ in which the sites of $A$ are empty, any
$x\in\left[L+1,\infty\right]^{d}$, and any $e'\in\{e_{1},\dots,e_{d}\}$,
the configuration $\eta^{x,x+e'}$ is connected to $\eta$.
\end{claim}

\begin{proof}
The only difference between this claim and Claim \ref{claim:russiandoll}
is that now $e'$ is chosen after $A$ is fixed. In order to achieve
that, we apply Claim \ref{claim:russiandoll} $d$ times, with $e'=e_{i}$
for all $i\in\{1,\dots,d\}$, obtaining $d$ numbers $L_{1},\dots,L_{d}\in\nn$
and $d$ sets $A_{1}\in\Lambda_{L_{1}},\dots,A_{d}\in\Lambda_{L_{d}}$.
Taking $L=\max_{i}L_{i}$ and $A=\cup_{i=1}^{d}A_{i}$ will suffice
-- fix $\eta$ in which the sites of $A$ are empty, every $x\in\left[L+1,\infty\right]^{d}$
and $i\in\{1,\dots,d\}$. In particular $x\in\left[L_{i}+1,\infty\right]^{d}$,
and that the sites of $A_{i}$ are empty in $\eta$, so by construction
of $A_{i}$ we know that $\eta^{x,x+e_{i}}$ is connected to $\eta$.
\end{proof}
We are now ready to prove Proposition \ref{prop:infinitereachable_noncooperative}.
\begin{proof}[Proof of Proposition \ref{prop:infinitereachable_noncooperative}]
We assume that for all $e\in\{\pm e_{1},\dots,\pm e_{d}\}$ there
exists a finite set $A_{e}\subset\zz^{d}$ such that the $e$-stretch
of $\eta_{A_{e}}$ is infinite, and construct a mobile cluster $A$.

First, use Claim \ref{claim:flipping_edges} in order to find $L_{+}\in\nn$
and $A_{+}\subset\Lambda_{L_{+}}$ such that for any $\eta$ in which
the sites of $A_{+}$ are empty, any $x\in\left[L_{+}+1,\infty\right]^{d}$,
and any $e\in\{e_{1},\dots,e_{d}\}$, the configuration $\eta^{x,x+e}$
is connected to $\eta$. Similarly (by flipping $\zz^{d}$), we can
find $L_{-}\in\nn$ and $A_{-}\subset\Lambda_{L_{-}}$ such that for
any $\eta$ in which the sites of $A_{-}$ are empty, any $x\in\left[-\infty,-L_{-}-1\right]^{d}$,
and any $e\in\{-e_{1},\dots,-e_{d}\}$, the configuration $\eta^{x,x+e}$
is connected to $\eta$. It will be more convenient to consider translations
of these sets, 
\begin{align*}
A'_{+} & =A_{+}-(L_{+}+2)e_{1}-\dots-(L_{+}+2)e_{d},\\
A'_{-} & =A_{-}+(L_{-}+2)e_{1}+\dots+(L_{-}+2)e_{d}.
\end{align*}
This way, for any $\eta$ in which the sites of $A_{+}'$ are empty,
any $x\in\left[2,\infty\right]^{d}$, and any $e\in\{\pm e_{1},\dots,\pm e_{d}\}$,
the configuration $\eta^{x,x+e}$ is connected to $\eta$; and for
any $\eta$ in which the sites of $A_{-}'$ are empty, any $x\in\left[-\infty,-2\right]^{d}$,
and any $e\in\{\pm e_{1},\dots,\pm e_{d}\}$, the configuration $\eta^{x,x+e}$
is connected to $\eta$. Let 
\[
A=A'_{+}\cup A_{-}'
\]
We will show that it is a mobile cluster. Since already $A_{+}'$
allows us to flip edges is its vicinity, we only need to show that
$\eta_{A}$ is connected to $\eta_{e+A}$ for all $e\in\{\pm e_{1},\dots,\pm e_{d}\}$.
To do that, we note that, since the sites of $A_{-}'$ are all in
$\left[2,\infty\right]$, the configuration $\eta_{A}$ is connected
to $\eta_{A'_{+}\cup(e+A_{-}')}$. In this new configuration the sites
of $e+A_{-}'$ are empty, and since the sites of $A_{+}'$ are all
in $\left[-\infty,-2\right]^{d}+e$ it is connected to $\eta_{(e+A_{+}')\cup(e+A_{-}')}=\eta_{e+A}$.
\end{proof}
\bibliographystyle{plain}
\bibliography{ka_HL}
\vspace{1cm}
\end{document}